\documentclass[12pt, a4paper]{article}
\usepackage{amssymb,amsmath,latexsym,color,graphicx}
\usepackage{hyperref}
\newcommand \clb{\color{blue}}    
\newcommand \clr{\color{red}}     
\newcommand \clm {\color{magenta}}

\newcommand \for {\mbox{ for } }
\newtheorem{theorem}{Theorem}[section]
\newtheorem{lemma}[theorem]{Lemma}
\newtheorem{proposition}[theorem]{Proposition}

\newtheorem{property}[theorem]{Property}

\newtheorem{definition}[theorem]{Definition}

\newtheorem{remark}[theorem]{Remark}

\newenvironment{proof}{\par\noindent{{\bf Proof.}}}{\hfill$\Box$
\medskip}

\newtheorem{prop}[theorem]{Proposition}

\title{Improving semi-groups bounds with resolvent estimates}
\author{B.~Helffer\footnote{Bernard.Helffer@univ-nantes.fr}\\Laboratoire de Math\'ematiques Jean Leray,\\ Nantes Universit\'e and
  CNRS,\\ \and\\ J.~Sj\"ostrand\footnote{Johannes.Sjostrand@u-bourgogne.fr}\\
Institut de Math\'ematiques de Bourgogne, UMR 5584 CNRS,\\  Universit\'e Bourgogne Franche-Comt\'e,\\ 
F21000 Dijon Cedex France.}

\date{\today}
\begin{document}

\bibliographystyle{plain}

\maketitle
\begin{abstract}
The purpose of this paper is to revisit the proof of the
Gearhardt-Pr\"uss-Hwang-Greiner theorem for a semigroup $S(t)$, following the
general idea of the proofs that we have seen in the literature and to
get an explicit estimate on $\Vert S(t)
\Vert$ in terms of bounds on the resolvent of the generator. A first version of this paper 
 was presented by the two authors  in ArXiv (2010)  together with applications in semi-classical analysis   and a part of these results  has been published later  in two books written by the authors. Our aim is to present new improvements, partially motivated by a paper of D. Wei.  On the way we discuss optimization problems confirming the optimality of our results.
\end{abstract}
{\color{blue}\tableofcontents }
\section{Introduction}\label{int}\setcounter{equation}{0}

Let ${\cal H}$ be a complex Hilbert space and let $[0,+\infty [\ni
t\mapsto S(t)\in {\cal L}({\cal H},{\cal H})$ be a strongly continuous
semigroup with $S(0)=I$. Recall that by the Banach-Steinhaus theorem,
$\sup_J\Vert S(t)\Vert=:m(J)$ is bounded for every compact
interval $J\subset [0,+\infty [$. Using the semigroup property it
follows easily that there exist $M\ge 1$ and $\omega_0 \in \mathbb{R}$
such that $S(t)$ has the property 
\begin{equation}\label{int.1}P(M,\omega_0 ):\quad \Vert S(t)\Vert\le Me^{\omega_0 t},\ t\ge
0.
\end{equation}
\par Let $A$ be the generator of the semigroup (so that formally $S(t)=\exp tA$) and recall
(cf. \cite{EnNa07}, Chapter II or \cite{Paz}) that $A$ is closed and densely
defined. We also recall (\cite{EnNa07}, Theorem II.1.10) that 
\begin{equation}\label{int.2}
(z-A)^{-1}=\int_0^\infty S(t)e^{-tz}dt,\quad \Vert (z-A)^{-1}\Vert \le
\frac{M}{\Re z-\omega_0}\,,
\end{equation}
when $P(M,\omega_0 )$ holds and $z$ belongs to the open half-plane $\Re z> \omega_0 $. 

According to the Hille-Yosida theorem (\cite{EnNa07}, Th.~II.3.5),
the following three statements are equivalent when $\omega
\in \mathbb{R}$:
\begin{itemize}
\item $P(1,\omega )$ holds.
\item $\Vert (z-A)^{-1}\Vert \le (\Re z-\omega )^{-1}$, when $z\in
  \mathbb{C}$ and $\Re z> \omega $.
\item $\Vert (\lambda -A)^{-1}\Vert \le (\lambda -\omega )^{-1}$, when
  $\lambda \in ]\omega ,+\infty [$.
\end{itemize}
Here we may notice that we get from the special case $\omega =0$ to
general $\omega $ by passing from $S(t)$ to
$\widetilde{S}(t)=e^{-\omega t}S(t)$.

Also recall that there is a similar characterization of the property
$P(M,\omega )$ when $M>1$, in terms of the norms of all powers of the
resolvent. This is the Feller-Miyadera-Phillips theorem (\cite{EnNa07},
Th.~II.3.8). Since we need all powers of the resolvent, the practical
usefulness of that result is less evident. 

\par We next recall the Gearhardt-Pr\"uss-Hwang-Greiner theorem, see
\cite{EnNa07}, Theorem V.I.11, \cite{TrEm}, Theorem 19.1:
\begin{theorem}\label{int1}~
\par\noindent 
(a) Assume that $\Vert (z-A)^{-1}\Vert$ is uniformly bounded in the
half-plane $\Re z\ge \omega $. Then there exists a constant $M>0$ such
that $P(M,\omega )$ holds.\\
(b) If $P(M,\omega )$ holds, then for every $\alpha >\omega $, $\Vert
(z-A)^{-1}\Vert$ is uniformly bounded in the half-plane $\Re z\ge
\alpha $.
\end{theorem}

\par The purpose of this paper is to revisit the proof of (a), following the
general idea of the proofs that we have seen in the literature and to
get an explicit $t$ dependent estimate on $e^{-\omega t}\Vert S(t)
\Vert$, implying explicit bounds on $M$. \\

 This idea is essentially to use that the resolvent and
  the inhomogeneous equation $(\partial _t-A)u=w$ in exponentially
  weighted spaces are related via
  Fourier-Laplace transform and we can use Plancherel's formula. Variants of
  this simple idea have also been used in more concrete
  situations. See \cite{BuZw, GGN, Hi, Sch09}  and a very complete
  overview of the possible applications in \cite{CST}. In this
    paper, we will obtain general results of the form:

    \medskip\par\noindent {\it If $\|
      S(t)\|\le m(t)$ for some positive function $m$, and if we have a
      certain bound on the resolvent of $A$, then $\|
      S(t)\|\le \widetilde{m}(t)$ and hence $\|
      S(t)\|\le \min (m(t),\widetilde{m}(t))$ for a new function
      $\widetilde{m}$ that can be explicitly described.}\\

Note that we can extend the conclusion of (a). If the
property (a) is true for some $\omega$ then it is automatically true
 for some  $\omega '<\omega $. We recall indeed 
  the following
\begin{lemma}\label{int01}~\\ If for some $r(\omega )>0$, $\Vert (z-A)^{-1}\Vert \le \frac{1}{r(\omega )}$
  for $\Re z>\omega $, then for every $\omega '\in ]\omega
  -r(\omega ),\omega ]$ we have 
$$
\Vert (z-A)^{-1}\Vert \le \frac{1}{r(\omega )-(\omega -\omega
  ')},\ \Re z>\omega '.
$$
\end{lemma} 
\medskip
Let
$$\omega _1=\inf\{\omega \in \mathbb R ; \, \{z\in \mathbb C ; \Re z>\omega
\}\subset \rho (A) \hbox{ and }\sup_{\Re z>\omega
}\Vert (z-A)^{-1}\Vert <\infty \}.$$
For $\omega >\omega _1$, we may define $r(\omega )$ by
\begin{equation}\label{defr}
\frac{1}{r(\omega )}=\sup_{\Re z>\omega }\Vert (z-A)^{-1}\Vert.
\end{equation}
Then $r(\omega )$ is an increasing function of $\omega $;  for every $\omega \in ]\omega _1,\infty [$, we have $\omega
-r(\omega )\ge \omega _1$ and for $\omega '\in [\omega -r(\omega
),\omega ]$ we have 
$$
r(\omega ')\ge r(\omega )-(\omega -\omega ').
$$

\begin{remark}~\\
Under the assumption $P(M,\omega _0)$ in (\ref{int.1}), we already
know from (\ref{int.2}) that\\
$\Vert (z-A)^{-1}\Vert$ is uniformly bounded in the half-plane $\Re
z\ge \beta $, if $\beta >\omega _0$. If $\alpha \le \omega _0$, we
see that $\Vert (z-A)^{-1}\Vert$ is uniformly bounded in the
half-plane $\Re z\ge \alpha $, provided that 
\begin{itemize}
\item we have this uniform boundedness on the line $\Re z=\alpha $,
\item $A$ has no spectrum in the half-plane $\Re z\ge \alpha $,
\item $\Vert (z-A)^{-1}\Vert$ does not grow too wildly in the strip
  $\alpha \le \Re z \le \beta $: $$\Vert (z-A)^{-1}\Vert \le {\cal
    O}(1) \exp ({\cal O}(1)\exp (k|\Im z|))\,,$$ where $k<\pi /(\beta
  -\alpha )$.

\end{itemize}
We then also have
\begin{equation}\label{int.3}
\sup_{\Re z\ge \alpha }\Vert (z-A)^{-1}\Vert = \sup_{\Re z=\alpha
}\Vert (z-A)^{-1}\Vert .
\end{equation}
This follows from the subharmonicity of $\log  ||(z-A)^{-1}||$, basically
Hadamard's theorem (or the one of Phragm\'en-Lindel\"of in
exponential coordinates). 
\end{remark}

\par The main result in \cite{HelSj} was:
\begin{theorem}\label{int2}~\\
We make the assumptions of Theorem \ref{int1}, (a) and let $r(\omega
)>0$ be as in (\ref{defr}). 
Let $m(t)\ge \Vert S(t)\Vert$ be a continuous positive function.
 Then
for all $t,a,b>0$, such that $t \geq  a+b$, we have
\begin{equation}\label{int.4}  \Vert S(t)\Vert \le \frac{e^{\omega t}} {r(\omega )\Vert
    \frac{1}{m}\Vert_{e^{-\omega \cdot }L^2(0,a)}\Vert
    \frac{1}{m}\Vert_{e^{-\omega \cdot }L^2(0,b)}}.
    \end{equation}
\end{theorem}
Here the norms are always the natural ones obtained from ${\cal H}$,
$L^2$, thus for instance 
$\Vert S(t)\Vert= \Vert S(t)\Vert_{{\cal L}({\cal H},{\cal H})}$, if
$u$ is a function on $\mathbb{R}$ with values in $\mathbb{C}$ or in
${\cal H}$, $\Vert u\Vert$  denotes the natural $L^2$ norm, when
the norm is taken over a subset $J$ of $\mathbb{R}$, this is indicated
with a ``$L^2(J)$''. In (\ref{int.4}) we also have the natural norm
in the exponentially weighted space $e^{-\omega \cdot }L^2(0,a)$ and
similarly with $b$ instead of $a$; $\Vert
f\Vert_{e^{-\omega \cdot }L^2(0,a)}=\Vert e^{\omega \cdot }f(\cdot
)\Vert_{L^2(0,a)}$. \\

The proof of these theorems was first presented in \cite{HelSj} and
later published in the books of the authors.
In \cite{W}, Dongyi Wei, motivated by our first version \cite{HelSj} has proved the following theorem:\\
\begin{theorem}\label{th3.2}
 Let $H=-A$ be an $m$-accretive operator in a Hilbert space $\mathcal H$. Then we have, 
\begin{equation}\label{eq:w}
||S(t) || \leq e^{- r(0)t +\frac \pi 2}\,,\, \forall t\geq 0\,.
\end{equation}
\end{theorem}
Our aim is to deduce and improve these two theorems as a consequence
of a unique basic estimate that we present now.  Let $\Phi$ satisfy
\begin{equation}\label{defpropchi}
0\le \Phi \in C^1([0,+\infty [) \mbox{  with } \Phi (0)=0 \mbox{ and }  \Phi (t)>0
\mbox{ for } t>0\,,
\end{equation}
and assume that $\Psi$ has the same properties.  
(By a density argument we can replace $C^1([0,+\infty [)$ in
(\ref{defpropchi}) by the space of locally Lipschitz functions on
$[0,+\infty [$.)  For $t>0$, let $\iota_t$
be the reflection with respect to $t/2$: $\iota _tu(s)=u(t-s)$. With
this notation, we have the following theorem.
 \begin{theorem}\label{Th1.2} Under the assumptions of Theorem
   \ref{int2},  for any $\Phi$ and $\Psi$ satisfying
   \eqref{defpropchi} and for any $\epsilon_1,\epsilon_2 \in \{-,+\}$, we have 
\begin{equation}\label{LL.5}
 || S (t)|| _{\mathcal L(\cal H)}\le  e^{\omega t}\frac{\|(r(\omega
  )^2\Phi^2 -\Phi '^2)^{\frac{1}{2}}_-\,m\|_{e^{\omega \cdot }L^2([0,t[) } \|(r(\omega)^2 \Psi
  ^2- \Psi '^2)^{\frac{1}{2}}_-\, m\|_{e^{\omega \cdot }L^2([0,t[) }}
{\int_0^t (r(\omega
    )^2\Phi^2 -\Phi '^2)^{\frac{1}{2}}_{\epsilon_1}(r(\omega )^2\iota_t \Psi^2-\iota_t \Psi
    '^2)^{\frac{1}{2}}_{\epsilon_2} ds}\,.
\end{equation}
\end{theorem}
Here for $a\in \mathbb R$, $a_+=\max (a,0)$ and  $a_-= \max (-a,0)$.

We now discuss the consequences of this theorem that can be obtained with suitable choices of $\Phi,\Psi,\epsilon_1,\epsilon_2$.\\
The first one is a Wei like version of our previous Theorem
\ref{int2}.
\begin{theorem}\label{th1.7}
For positive $a$ and $b $,  we have, for $ t > a+b \,$, 
 \begin{equation}\label{eq:gp1}
 || S(t)|| \leq \frac{e^{\omega t - r(\omega) (t- a-b)} }{r(\omega)} \, \frac{1}{\Vert
    \frac{1}{m}\Vert_{e^{-\omega \cdot }L^2(0,a)} \Vert
    \frac{1}{m}\Vert_{e^{-\omega \cdot }L^2(0,b)} }. 
    \end{equation}
 \end{theorem}
In the case of Wei's theorem we have $\omega =0$, $m=1$. With $b=a$ we
first get
$$
 || S(t)|| \leq \frac{1 }{a r(0)} \exp - r(0) (t- 2a) \, ,\ \ t>2a.
 $$
Minimization with respect to $a$ leads to $ar(0)=\frac 12$ and consequently to 
 $$
  || S(t)|| \leq 2 e \exp - r(0) t\,,\ \ t>\frac{1}{r(0)}\, ,
 $$
 which is not quite as sharp as \eqref{eq:w},  since $e^{\pi
     /2}\approx 4.81$, $2e\approx 5.44$.\\
  
We will show that a finer approach will permit to recover \eqref{eq:w}
 and generalize it to more general $m$'s. We assume 
 \begin{equation}\label{eq:hypsurm}
 0 < m \in C^1([0,+\infty[)\,.
 \end{equation}
 
  An important step will be to prove (we assume $\omega=0$, $r(0)=1$) as a consequence of Theorem \ref{Th1.2} with $\epsilon_1 =-$ and $\epsilon_2=+$, the following key proposition 
 \begin{proposition}\label{propminmax} Assume that $\omega =0$,
   $r(\omega )=1$.
  Let $a, b$ positive. Then for  $t\geq a+b$,  
 \begin{equation}\label{eq:optprob}
 || S(t)|| \leq \exp -( t -a - b ) \,\frac{ \left(\inf_u \int_0^a
     m(s)^2 (u'^2(s)-u^2(s))_+ ds\right)^{1/2}  }{  \left( \sup_\theta \int_0^{b } \frac{1}{m^2}  (\theta(s)^2-\theta'(s)^2) \,ds \right)^{1/2} } \,,
 \end{equation}
 where
 \begin{itemize}
 \item $u\in H^1(]0,a[)$ satisfies $u(0)=0$, $u(a) =1$\,;
 \item $\theta \in H^1((]0,b[)$ satisfies $\theta(b) = 1$ and $ |\theta '|\le \theta$\,.
 \end{itemize}
 \end{proposition}
 This proposition implies rather directly Theorem \ref{th1.7} in the
 following way. We first observe the trivial lower bound (take
 $\theta(s)=1$)
  \begin{equation}\label{eq:gp2} \sup_\theta \int_0^{b } \frac{1}{m^2}  (\theta(s)^2-\theta'(s)^2) \,ds \geq  \int_0^{b } \frac{1}{m^2} ds\,.
  \end{equation}
  
  A more tricky argument based on the equality case in Cauchy-Schwarz'
  inequality (see Subsection \ref{ss3.6} for details), gives
\begin{equation}\label{eq:gp3}
 \inf_u \int_0^a m(s)^2 (u'(s)^2-u(s)^2)_+ ds  \leq  \inf_u \int_0^a m(s)^2 u'(s)^2 ds  \leq 1/  \int_0^{a} \frac{1}{m^2} ds\,,
 \end{equation}
 Combining \eqref{propminmax} with \eqref{eq:gp2} and \eqref{eq:gp3}
 gives directly \eqref{eq:gp1} in the case $\omega =0$,
 $r(\omega )=1$.  A rescaling argument (which will be detailed in
 Subsection \ref{ss4.1new}) then gives (\ref{eq:gp1}) in
 general.\\

 To refine the analysis of the right hand side of
 \eqref{eq:optprob}, we have to analyze for positive $a$ and $b$ the
 quantities
 $$ I_{\mathrm{inf}}(a):=  \inf_u \int_0^a m(s)^2 (u'(s)^2-u^2(s))_+ ds $$ and $$J_{\mathrm{max}} (b):= \sup_\theta \int_0^{b } \frac{1}{m^2}  (\theta(s)^2-\theta'(s)^2) \,ds\,,
 $$
 where $u$ and $\theta $ satisfy the above conditions.  This will be
 the main object of Section \ref{s3}.  
 To present some of the results
 in this introduction, we consider the Dirichlet-Robin realization
 $K_{m,a}^{DR}$ of the operator
 \begin{equation}\label{eq:defK}
 K_m := -\frac{1}{m^2} \partial_s \circ m^2 \partial_s - 1\,,
 \end{equation}
 in the interval $]0,a[$. The Dirichlet-Robin condition is
 \begin{equation}\label{DRcond}
 u(0) =0\,,\, u'(a) = u(a)\,,
\end{equation}
and we define the domain of $K_{m,a}^{DR}$ by
$$
{\cal D}(K_{m,a}^{DR})=\{ u\in H^2(]0,a[);\, u\hbox{ satisfies
  (\ref{DRcond})} \} .
$$
 We note that  this realization is a  self-adjoint operator
 on $L^2(]0,a[, m^2 ds)$, bounded from below and  with purely discrete spectrum.

 \par Let $\lambda ^{DR}(a,m)$ denote the lowest eigenvalue of
 $K_{m,a}^{DR}$. Then $\lambda ^{DR}(a,m)>0 $ when $a>0$ is small
 enough. We define
 \begin{equation}\label{defa*}
a^*=a^*(m)=\sup \{\widetilde{a}\in ]0,\infty [; \lambda
^{DR}(a,m)>0\hbox{ for }0<a<\widetilde{a} \},
\end{equation}
so that $a ^*(m)\in ]0,+\infty ]$. Since $\lambda ^{DR}(a,m)$ is a
continuous function of $a$, we have in the case $a^*<\infty $ that
$$
\lambda ^{DR}(a^*,m)=0,\ \ \lambda ^{DR}(a,m)>0 \hbox{ for }0<a<a^*.
$$

\par We  introduce the condition 
 \begin{equation}\label{eq:C1weak}
 \liminf_{s\rightarrow +\infty} \mu(s) > - 1\mbox{ with } \mu := m'/m\,.
 \end{equation}
  
 Under this condition, we will show   that $a^*(m)<+\infty $.
We will show  in Section \ref{s3} that if
  on $]0,a^{*} [$
  \begin{equation}
  \psi_0 (s;m)=\psi_0: = u'_0(s) / u_0 (s)\,,\ \ 0<s<a^*,
  \end{equation}
  where $u_0 $ is the first eigenfunction of the $DR$-problem in
  $]0,a^{*} [$, then: 
    \begin{theorem}\label{prop4.9} Let $\omega =0$, $r(\omega )=1$. When
    $a,b\in ]0,+\infty [\cap ]0,a^*]$
  and $t> a+b$, we have 
 \begin{equation}
 || e^t S(t)|| \leq \exp (a +b ) m(a) m(b) \psi_0(a)^\frac 12  \psi_0 (b)^\frac 12  \,.
\end{equation}
In particular, when $a^*<+\infty $, we have 
 \begin{equation}\label{eq:thintro}
 || e^t S(t)|| \leq \exp (2 a^{*})   \, m(a^{*} )^2 \,,\ \ t>2 a^*\,.
 \end{equation}
\end{theorem}
 This theorem is the analog of Wei's theorem for general  weights $m$.\\
 By a general procedure described in Subsection \ref{ss4.1new}, we
 have actually a more general statement.  We consider $\hat A$ with the same properties as $A$ where the hat's are introduced to make easier the transition between the particular case above to the general case below. As before, we introduce
   $\hat \omega $ and $\hat r =\hat r(\hat \omega)$. 
    \begin{theorem}\label{prop4.9bis} 
    Let $\hat S(\hat t)= e^{\hat t \, \hat A} $ satisfying $$
    || \hat S (\hat t) || \leq \hat m(\hat t)\,,\, \forall \hat t >0\,.
    $$
    Then there exist uniquely defined\footnote{The definition will be given in Subsection \ref{ss4.1new}}  $\hat a^*:=\hat a^*(\hat m, \hat \omega,\hat r )>0 $ and $\hat \psi:=\hat \psi (\cdot ; \hat m, \hat \omega,\hat r )$ on $]0,\hat a^*[$ with the same general properties as above  such that,
 if  $\hat a, \hat b\in ]0,+\infty [\cap ]0,\hat a^*]$
  and $\hat t> \hat a+\hat b$, we have 
 \begin{equation} 
 || S(\hat t)|| \leq \exp \left( ( \hat \omega  -  \hat r(\omega)) (\hat t- (\hat a + \hat b )) \right)  \hat m(\hat a) \hat m(\hat b)   \hat \psi(\hat a)^\frac 12 \, \hat \psi(\hat b )^\frac 12   \,.
\end{equation}
Moreover, when $ \hat a^*<+\infty $, the estimate is optimal for $\hat a=\hat b=\hat a^{*} $ and reads 
 \begin{equation}\label{eq:thintrobis}
 ||  \hat S(\hat t)|| \leq \exp ((\hat \omega   -  \hat r ) (\hat t - 2\hat a^{*}))   \, \hat m(\hat a^*)^2 \,,\ \ t>2 \hat a^*\,.
 \end{equation}
 Moreover
 $$
 \hat a^* (\hat m,\hat \omega)= \hat r \, a^* (e^{-\hat \omega \cdot}\hat m) \,,\, \hat \psi (\hat s;  \hat m, \hat \omega,\hat r ) = \psi_0 (\hat r \hat s ;  e^{-\hat \omega \cdot}\hat m)   \,.
$$
\end{theorem}

 Theorem \ref{th1.7}, Proposition \ref{propminmax} and Theorem
  \ref{prop4.9} are based   in Section \ref{s4}  on Theorem \ref{Th1.2}, with the choice
  $(\epsilon _1,\epsilon _2)=(+,-)$  which is proved in Section \ref{s2}.  In the appendix we explore the
  consequences of the choice $(\epsilon _1,\epsilon _2)=(+,+)$. In
  this case it turned out to be more difficult to reach equally clear
applications. 
  
 \section{Proof of Theorem \ref{Th1.2}}\label{s2} \setcounter{equation}{0}
\subsection{Flux} Let
$u(t)\in C^1([0,+\infty [;{\cal H})\cap C^0([0,+\infty [;{\cal
  D}(A))$,
$u^*(t)\in C^1(]-\infty ,T];{\cal H})\cap C^0(]-\infty ,T];{\cal
  D}(A^*))$, solve $(A-\partial _t)u=0$ and $(A^*+\partial _t)u^*=0$
on $[0,+\infty [$ and $]-\infty ,T]$ respectively. Then the flux (or
Wronskian) $[u(t)|u^*(t)]$ is constant on $[0,T]$ as can be seen by
computing the derivative with respect to $t$. Here we use the
notations $[\cdot |\cdot ]_{\cal H}$ and $|\cdot |_{\cal H}$ for the
``point-wise'' scalar product and norm in ${\cal H}$.

\subsection{$L^2$ estimate} 
Write $L^2_\phi (I)=L^2(I;e^{-2\phi }dt)=e^\phi L^2(I)$, $\|u\|_\phi
=\|u\|_{\phi ,I}=\|u\|_{L^2_\phi (I)}$, where $I$ is an interval and our functions take values in ${\cal
  H}$. (Our vector valued functions
will be norm continuous, so we avoid the formal definition of
these spaces with the Lebesgue integral and manage with the Riemann integral.)
By Parseval-Plancherel, the Laplace transform
$$
{\cal L}u(\tau )=\int e^{-t\tau }u(t)dt
$$
gives a unitary map from $L^2_{\omega \cdot }(\mathbb R)$ to
$L^2(\Gamma _\omega ;d\Im \tau /(2\pi ))$, where $\Gamma _\omega
\subset \mathbb C$ denotes the line given by $\Re \tau =\omega $ and
$\omega $ is real. 
By applying ${\cal L}$ we see that
$(A-\partial _t)^{-1}:\, L^2_{\omega \cdot }(\mathbb R)\to L^2_{\omega
  \cdot }(\mathbb R)$ is well-defined and bounded of norm
$1/r(\omega )$.
\par Consider $(A-\partial _t)u=0$ on $[0,+\infty [$ with
$u\in L^2_{\omega \cdot }([0,+\infty [)$. \\
Let $\Phi$ satisfy (\ref{defpropchi}) and add temporarily the
assumption that $\Phi (s)$ is constant for $s\gg 0$. 
 Then $\Phi u$, $\Phi 'u$ can be viewed as elements of
$L^2_{\omega \cdot }(\mathbb R)$ and from
$$
(A-\partial _t)\Phi u=-\Phi 'u\,,
$$
we get, by the definition of $r(\omega)$,  
$$
\| \Phi u\|_{\omega \cdot }\le \frac{1}{r(\omega )}\|\Phi
'u\|_{\omega \cdot },
$$
or, taking the square,
$$ 
((r(\omega )^2\Phi^2 -\Phi '^2)u|u)_{\omega \cdot }\le 0\,.
$$
This can be rewritten as 
\begin{equation}\label{L2.1}
  ((r(\omega )^2\Phi^2 -\Phi '^2)_+u|u)_{\omega \cdot }\le
  ((r(\omega )^2\Phi^2 -\Phi '^2)_-u|u)_{\omega \cdot },
\end{equation}
or
\begin{equation}\label{L2.2}
\| (r(\omega )^2\Phi^2 -\Phi '^2)_+^{1/2}u\|_{\omega \cdot }\le 
\| (r(\omega )^2\Phi^2 -\Phi '^2)_-^{1/2}u\|_{\omega \cdot }.
\end{equation}
By a limiting procedure, we see that (\ref{L2.1}), (\ref{L2.2}) remain
valid without the assumption that $\Phi $ be constant near $+\infty $.\\
Writing $\Phi =e^{\phi }$, $\phi \in C^1(]0,+\infty [)$, $\phi (t)\to
-\infty $ when $t\to 0$, we have
$$
r(\omega )^2\Phi ^2-\Phi '^2=(r(\omega )^2-\phi '^2)e^{2\phi } \,,
$$
and (\ref{L2.1}), (\ref{L2.2}) become
\begin{equation}\label{L2.3}
  ((r(\omega )^2-\phi '^2)_+u|u)_{\omega \cdot -\phi  }\le
  ((r(\omega )^2-\phi '^2)_-u|u)_{\omega \cdot -\phi  }\,,
\end{equation}
\begin{equation}\label{L2.4}
  \|(r(\omega )^2-\phi '^2)_+^{1/2}u\|_{\omega \cdot -\phi  }\le
  \|(r(\omega )^2-\phi '^2)_-^{1/2}u\|_{\omega \cdot -\phi  }\,.
\end{equation}
We have in mind the case when $r(\omega )^2-(\phi ')^2>0$ away from a
bounded neighborhood of $t=0$.

\par Let $S(t)=e^{tA}$, $t\ge 0$ and let $m(t)>0$ be a continuous function such
that
\begin{equation}\label{L2.5}
\| S(t)\|\le m(t),\ t\ge 0.
\end{equation}
Then we get
\begin{equation}\label{L2.6}
  \| (r(\omega )^2-\phi '^2)_+^{1/2}u\|_{\omega \cdot -\phi  }
  \le \| (r(\omega )^2-\phi '^2)_-^{1/2}m\|_{\omega \cdot -\phi 
  }|u(0)|_{\cal H}.
\end{equation}
 Note that we have also trivially
\begin{equation}\label{L2.6bis}
  \| (r(\omega )^2-\phi '^2)_-^{1/2}u\|_{\omega \cdot -\phi  }
  \le \| (r(\omega )^2-\phi '^2)_-^{1/2}m\|_{\omega \cdot -\phi 
  }|u(0)|_{\cal H}.
\end{equation}

\par We get the same bound for the forward solution of $A^*-\partial
_t$ and, after changing the orientation of time, for the backward
solution of $A^*+\partial _t=(A-\partial _t)^*$. Then for $u^*(s)$,
solving
$$
(A^*+\partial _s)u^*(s)=0,\ s\le t,
$$
with $u^*(t)$ prescribed, we get
\begin{equation*}
  \| (r(\omega )^2-\iota_t\phi '^2)_+^{1/2}u^*\|_{\omega (t-\cdot )-\iota_t \phi }
  \le
  \| (r(\omega )^2-\iota_t \phi '^2)_-^{1/2}\iota_t m
  \|_{\omega (t-\cdot )-\iota_t \phi }\,  |u^*(t)|_{\mathcal H}\,,
\end{equation*}
where $\iota_t \phi $ and $\iota_t m $ denote the
compositions of $\phi $ and $m$ respectively with the reflection $\iota_t$  in
$t/2$ so that $$\iota_t m (s)=m(t-s),\ \ \iota_t \phi (s)= \phi (t-s)\,.$$

 More generally, we can replace $\phi$ by $\psi$ with the same properties (see \eqref{defpropchi})
and consider $ \Psi  = \exp \psi\,.
$

 Note that we have 
\begin{equation}\label{L2.7}
  \| (r(\omega )^2-\iota_t \psi '^2)_+^{1/2}u^*\|_{ \omega (t-\cdot
    ) -\iota_t \psi  }
  \le \| (r(\omega )^2-\psi'^2)_-^{1/2}m\|_{\omega \cdot -\psi 
  }|u^*(t)|_{\cal H}.
\end{equation}
and  also trivially
\begin{equation}\label{L2.7bis}
  \| (r(\omega )^2-\iota_t \psi '^2)_-^{1/2}u^*\|_{ \omega (t-\cdot )-\iota_t \psi }
  \le \| (r(\omega )^2-\psi'^2)_-^{1/2}m\|_{\omega \cdot -\psi 
  }|u^*(t)|_{\cal H}.
\end{equation}

\subsection{From $L^2$ to $L^\infty $ bounds}\label{LL} 
 In order to estimate
$|u(t)|_{{\cal H}}$ for a given $u(0)$ it suffices to estimate
$[u(t)|u^*(t)]_{\cal H}$ for arbitrary $u^*(t)\in {\cal H}$. Extend
$u^*(t)$ to a backward solution $u^*(s)$ of \break $(A^*+\partial
_s)u^*(s)=0$, so that
$$
[u(s)|u^*(s)]_{\cal H}=[u(t)|u^*(t)]_{\cal H},\ \forall s\in [0,t]. 
$$

\par Let $M=M_t:[0,t]\to [0,+\infty [$ have mass 1:
\begin{equation}\label{LL.1}
\int_0^t M(s)ds=1.
\end{equation}
Then
\begin{equation}\label{LL.1.5}
  |[u(t)|u^*(t)]_{\cal H}|=\left| \int_0^t M(s)[u(s)|u^*(s)]_{\cal H} ds \right|
  \le \int_0^t M(s) |u(s)|_{\cal H}|u^*(s)|_{\cal H} ds. \end{equation}
 Let $\epsilon _1,\epsilon_2 \in \{-,+\}$.  Assume that
\begin{equation}\label{LL.2}
  \mathrm{supp\,}M\subset \{ s; \epsilon_1(r(\omega )^2-\phi '(s)^2) > 0,\
\epsilon_2(  r(\omega )^2-\iota_t\psi '(s)^2) > 0  \}. 
\end{equation}

Then multiplying and dividing with suitable factors in the last member
of (\ref{LL.1.5}), we get
 
\begin{equation*} 
\begin{split}
  \left| [u(t)|u^*(t)]_{\cal H} \right|& \le e^{\omega t}\int_0^t
  \frac{M(s)e^{-\phi (s)-\iota_t \psi (s)}} {(r(\omega )^2-\phi
    '(s)^2)^{\frac{1}{2}}_{\epsilon_1}(r(\omega )^2-\iota_t \psi
    '(s)^2)^{\frac{1}{2}}_{\epsilon_2}}\times\\ 
    & \qquad \times  e^{\phi (s)-\omega s} (r(\omega
  )^2-\phi '(s)^2)^{\frac{1}{2}}_{\epsilon_1}|u(s)|_{\cal H}\times \\ &\qquad \times  e^{\iota_t \psi
    (s)-\omega (t-s)}(r(\omega )^2-\iota_t \psi
  '(s)^2)^{\frac{1}{2}}_{\epsilon_2}|u^*(s)|_{\cal
    H}ds\\
     & \le
  e^{\omega t}\sup_{[0,t]} \frac{Me^{-\phi-\iota_t \psi }} {(r(\omega
    )^2-\phi '^2)^{\frac{1}{2}}_{\epsilon_1}(r(\omega )^2-\iota_t \psi
    '^2)^{\frac{1}{2}}_{\epsilon_2}}\times\\ & \qquad \times \|(r(\omega
  )^2-\phi '^2)^{\frac{1}{2}}_{\epsilon_1}u\|_{ \omega \cdot -\phi }\|(r(\omega )^2-\iota_t \psi
  '^2)_{\epsilon_2}^{\frac{1}{2}}u^*\|_{\omega (t-\cdot )-\iota_t \psi }.
  \end{split}
  \end{equation*}

Using (\ref{L2.6}), (\ref{L2.7}) when $\epsilon_j=+$ or   (\ref{L2.6bis}), (\ref{L2.7bis}) when $\epsilon_j=-$,   we get
\begin{equation*} 
\begin{split}
 \left| [u(t)|u^*(t)]_{\cal H} \right|& \le  e^{\omega t}\sup_{[0,t]} \frac{Me^{-\phi- \iota_t \psi }} {(r(\omega
    )^2-\phi '^2)^{\frac{1}{2}}_{\epsilon_1}(r(\omega )^2-\iota_t \psi
    '^2)^{\frac{1}{2}}_{\epsilon_2}}\times\\ &  \times 
  \|(r(\omega
  )^2-\phi '^2)^{\frac{1}{2}}_-m\|_{\omega \cdot -\phi  }\|(r(\omega )^2-\psi
  '^2)_-^{\frac{1}{2}}\, m\|_{\omega \cdot
-\psi     }|u(0)|_{\cal H}|u^*(t)|_{\cal H}.
 \end{split}
  \end{equation*}

Choosing $u^*(t)=u(t)$, gives
\begin{equation}\label{LL.3}
\begin{split}
  |u(t)|_{\cal H}& \le e^{\omega t}\sup_{[0,t]}
  \frac{Me^{-\phi-\iota_t \psi }} {(r(\omega )^2-\phi
    '^2)^{\frac{1}{2}}_{\epsilon_1}(r(\omega )^2-\iota_t \psi
    '^2)^{\frac{1}{2}}_{\epsilon_2}}\times\\ & \qquad \times
  \|(r(\omega)^2-\phi '^2)^{\frac{1}{2}}_-\, m\|_{\omega \cdot-\phi }
  \|(r(\omega)^2-\psi '^2)^{\frac{1}{2}}_-\, m\|_{ \omega
    \cdot -\psi }|u(0)|_{\cal H}.
\end{split}
\end{equation}
 In order to optimize the choice of $M$, we let $0\not\equiv F\in
 C([0,t];[0,+\infty [)$ and study
\begin{equation}\label{LL.4}
\inf_{0\le M\in C([0,t]),\atop \, \int Mds=1}\sup_s \frac{M(s)}{F(s)}.
\end{equation}
We first notice that
$$
1=\int Mds=\int \frac{M}{F}Fds\le \left(\sup_s \frac{M}{F} \right)\int Fds
$$
and hence
the quantity (\ref{LL.4}) is $\ge 1/\int Fds$. Choosing $M=\theta F$ 
with $\theta =1/\int F(s)\,ds$, we get equality.
 \footnote{ $M$  does not necessarily satisfy condition \eqref{LL.2} but we can proceeed via a limiting argument.}

\begin{lemma}\label{LL1} For any continuous function $F\geq 0$, non identically $0$,
$$
\inf_{0\le M\in C([0,t]),\atop \, \int M(s)\,ds=1}\  \left(\sup_s \frac{M}{F} \right)  = 1/\int Fds\,.
$$
\end{lemma}
Applying the lemma to the
supremum in (\ref{LL.3})  with  
$$F= e^{\phi+\iota_t \psi}\, (r(\omega
    )^2-\phi '^2)^{\frac{1}{2}}_{\epsilon_1}(r(\omega )^2-\iota_t \psi
    '^2)^{\frac{1}{2}}_{\epsilon_2},$$ 
we get 
\begin{equation}\label{LL.5a}
 |u(t)|_{\cal H} \le  e^{\omega t}\frac{\|(r(\omega
  )^2-\phi '^2)^{\frac{1}{2}}_-\, m\|_{\omega \cdot -\phi  } \|(r(\omega
  )^2- \psi '^2)^{\frac{1}{2}}_-\, m\|_{ \omega \cdot -\psi } }
{\int_0^t e^{\phi + \iota_t \psi }(r(\omega
    )^2-\phi '^2)^{\frac{1}{2}}_{\epsilon_1}(r(\omega )^2-\iota_t\psi
    '^2)^{\frac{1}{2}}_{\epsilon_2} ds}
  |u(0)|_{\cal H}.
\end{equation}
Since $u(0)$ is arbitrary, this is a rewriting of (\ref{LL.5}) and we get Theorem \ref{Th1.2}.
\begin{remark}\label{stupid}
If we do not impose any condition of the type (\ref{LL.2}), we get a
variant of Theorem \ref{Th1.2} which is easier to state, but probably
less sharp: Adding the squares of (\ref{L2.6}), (\ref{L2.6bis}), leads to
$$
\| |r(\omega )^2-\phi '^2|^{1/2}u\|_{\omega \cdot -\phi }\le \sqrt{2}
\| (r(\omega )^2-\phi '^2)_-^{1/2}m\|_{\omega \cdot -\phi }|u(0)|_{\mathcal{H}}
$$
Similarly, from (\ref{L2.7}), (\ref{L2.7bis}),
$$ 
\| |r(\omega )^2-\iota _t\psi '^2|^{1/2}u^*\|_{\omega (t-\cdot) -\psi }\le \sqrt{2}
\| (r(\omega )^2-\psi '^2)_-^{1/2}m\|_{\omega \cdot -\psi }|u^*(t)|_{\mathcal{H}}
$$
We then follow a simplified variant of the estimates after (\ref{LL.1.5}):
\begin{equation*} 
\begin{split}
  \left| [u(t)|u^*(t)]_{\cal H} \right|& \le e^{\omega t}\int_0^t
  \frac{M(s)e^{-\phi (s)-\iota_t \psi (s)}} {|r(\omega )^2-\phi
    '(s)^2|^{\frac{1}{2}}|r(\omega )^2-\iota_t \psi
    '(s)^2|^{\frac{1}{2}}}\times\\ 
    & \qquad \times  e^{\phi (s)-\omega s} |r(\omega
  )^2-\phi '(s)^2|^{\frac{1}{2}}|u(s)|_{\cal H}\times \\ &\qquad \times  e^{\iota_t \psi
    (s)-\omega (t-s)}|r(\omega )^2-\iota_t \psi
  '(s)^2|^{\frac{1}{2}}|u^*(s)|_{\cal
    H}ds\\
     & \le
  e^{\omega t}\sup_{[0,t]} \frac{Me^{-\phi-\iota_t \psi }} {|r(\omega
    )^2-\phi '^2|^{\frac{1}{2}}|r(\omega )^2-\iota_t \psi
    '^2|^{\frac{1}{2}}}\times\\ & \qquad \times \||r(\omega
  )^2-\phi '^2|^{\frac{1}{2}}u\|_{ \omega \cdot -\phi }\| |r(\omega )^2-\iota_t \psi
  '^2|^{\frac{1}{2}}u^*\|_{\omega (t-\cdot )-\iota_t \psi }\\
& \le 2 e^{\omega t}\sup_{[0,t]} \frac{Me^{-\phi- \iota_t \psi }} {|r(\omega
    )^2-\phi '^2|^{\frac{1}{2}}|r(\omega )^2-\iota_t \psi
    '^2|^{\frac{1}{2}}}\times\\ &  \times 
  \|(r(\omega
  )^2-\phi '^2)^{\frac{1}{2}}_-m\|_{\omega \cdot -\phi  }\|(r(\omega )^2-\psi
  '^2)_-^{\frac{1}{2}}\, m\|_{\omega \cdot
-\psi     }|u(0)|_{\cal H}|u^*(t)|_{\cal H}.
 \end{split}
  \end{equation*}
Choosing $u^*(t)=u(0)$ and applying Lemma \ref{LL1} gives the
following variant of (\ref{LL.5}),

\begin{equation}\label{LL.5var}\begin{split}
 &|| S (t)|| _{\mathcal L(\cal H)}\le \\  &2e^{\omega t}\frac{\|(r(\omega
  )^2\Phi^2 -\Phi '^2)^{\frac{1}{2}}_-\,m\|_{e^{\omega \cdot }L^2([0,t[) } \|(r(\omega)^2 \Psi
  ^2- \Psi '^2)^{\frac{1}{2}}_-\, m\|_{e^{\omega \cdot }L^2([0,t[) }}
{\int_0^t |r(\omega
    )^2\Phi^2 -\Phi '^2|^{\frac{1}{2}}\, |r(\omega )^2\iota_t \Psi^2-\iota_t \Psi
    '^2|^{\frac{1}{2}} ds}\,.
  \end{split}
\end{equation}
\end{remark}

Our goal is to show that starting from {\color{blue}(\ref{LL.5}),
  (\ref{LL.5a})} we can, by suitable choices of
$\Phi, \phi, \Psi,\psi, \epsilon_1,
\epsilon_2$, obtain and actually improve all the variants of
the previously obtained statements \cite{HelSj,W}. We will start by the analysis of two optimization problems which have their own independent interest.

\section{Optimizers}\label{s3}
\setcounter{equation}{0}
\subsection{Introduction}\label{ssint}
 Motivated  by Proposition  \ref{propminmax}, we study   in this section  the problem
of minimizing  an integral: 
\begin{equation}\label{red.1-intro}
 I_{\mathrm{inf}} (a) :=\inf_{\{u\in H^1(]0,a[);\, u(0)=0, u(a)=1\}} \int
_0^a (u'^2-u^2)_+m^2\, ds.
\end{equation}
 and of maximizing a similar
 integral:
 \begin{equation}\label{red.2-intro}
J_{\mathrm{sup}} (b):=\sup_{{\cal G}} \int
_0^b (\theta^2-\theta'^2)m^{-2} \,ds\,,
\end{equation}
where ${\cal G}$ is defined by
\begin{equation}\label{max.1}
{\cal G}=\{\theta \in H^1(]0,b[);\, |\theta '|\le \theta \mbox{ and } \ \theta
(b)=1 \}\, .
\end{equation}
The two problems are very similar, we devote most of the
section to the minimization problem in the  next four subsections and treat  more shortly the maximization problem in the last
Subsection \ref{ssmax}.
\subsection{Reduction}\label{red}
Let $0<m\in C^1([0,+\infty [)$.
If $0\le \sigma <\tau <+ \infty $ and $S,T\in \mathbb R$ we put
\begin{equation}\label{red.0.5}H_{S,T}^1(]\sigma ,\tau [)=\{ u\in H^1(]\sigma ,\tau [);\, u(\sigma
)=S,\ u(\tau )=T\}\,.\end{equation}
Here and in the following all functions are
assumed to be real-valued unless stated otherwise. In this section we
let $a\in ]0,+\infty [$ and 
study
\begin{equation}\label{red.1}
\inf_{u\in H^1_{0,1}(]0,a[)}I(u),\hbox{ where }I(u)=I_{]0,a[}(u)=\int
_0^a (u'^2-u^2)_+m^2ds.
\end{equation}
We shall show that we can here replace $H_{0,1}^1$ by a subspace that
allows to avoid the use of positive parts. Put
\begin{equation}\label{red.2}
{\cal H}={\cal H}_{0,a}^{0,1}\,,
\end{equation}
where for $\sigma ,\tau ,S,T$ as above,
\begin{equation}\label{red.3}
{\cal H}_{\sigma ,\tau }^{S,T}=\{u\in H_{S,T}^1(]\sigma ,\tau [); 0\le
u\le u'\}.
\end{equation}
Here the inequalities $0\le u$, $u\le u'$ are valid in the sense of
distributions, i.~e.  $u$ and $u'-u$ are positive distributions on
$]0,a[$.  Notice that if $S,T>0$ then for this space to be non-zero, it
is necessary that 
\begin{equation}\label{condnecST} 
T\ge e^{\tau -\sigma }S\,.
\end{equation}
\begin{prop}\label{red1}
  $$
\inf_{H_{0,1}^1}I(u)=\inf_{\cal H}I(u).
  $$
\end{prop}
\begin{proof} 
Clearly \begin{equation}\label{red.3.5}\inf_{H_{0,1}^1}I(u)\le \inf_{\cal H}I(u). \end{equation} We need to
establish the opposite inequality. 
\paragraph{Step 1.} 
We first show
\begin{equation}\label{red.4}
\inf_{u\in H_{0,1}^1}I(u)\ge \inf _{\{u\in H_{0,1}^1;\, u'\ge 0\}} I(u).
\end{equation}

In the left hand side, we can replace $H_{0,1}^1$ by the
dense subspace of Morse functions of class $C^2$  in $[0,a]$   where $0,1$ are not critical values, $u(0)=0$, 
$u(a)=1$.\\ 
We shall see that we can replace $u$ by a piecewise $C^1$ function\footnote{\label{fnpiecewise}
We say that $u=[0,a]\mapsto \mathbb R$ is piecewise $C^1$ if $u\in C^0([0,a])$ and $u'$ is piecewise continuous, i.e. with at most finitely many jump discontinuities. We denote by $C^1_{pw} ([0,a])$ the space of piecewise $C^1$ functions.} $v$  on $[0,a]$ with $v' \geq 0$,  $v(0)=0$, $v(a)=1$, s.t. $I(v)\leq I(u)$.\\
Let $M(u) \geq 0$ be the number of critical points of $u$ in $]0,a[$. If $M(u)=0$, then $u$ is increasing and we are done.\\
Assume that we can construct $v$ as above\footnote{ \label{dil_ext} Notice that
  by affine dilations in $s$, $u$ we  have the
  seemingly more general statement that if $\tilde u$ is a $C^2$ Morse
  function on $[\sigma,\tau]$, where
  $-\infty < \sigma < \tau < +\infty$,
  $\tilde u (\sigma) < \tilde u (\tau)$, and $\tilde u(\sigma)$,
  $\tilde u(\tau)$ are not critical values, then there is a piecewise
  $C^1$ function on $[\sigma,\tau]$, such that, $\tilde v' \geq 0$,
  $\tilde v(\sigma)=\tilde u (\sigma)$,
  $\tilde v(\tau)=\tilde u(\tau)$, and
  $I_{]\sigma,\tau[}(\tilde v) \leq I_{]\sigma,\tau[} (\tilde u)$.}
when $M(u) \leq M$ for some $M\in \mathbb N$,  and let us show
  that we can do the same when \begin{equation}\label{red.4.5}M(u)=M+1\end{equation} and we now consider that case.
\par 
Let $\sigma =\sup_{u(s)=0} s$. Then $u(\sigma)=0$ and $u(s) >0$ for $s>\sigma$.
If $\sigma >0$, then $u$ has at least one critical point in
$]0,\sigma[$ and hence $u$ has at most  $ M$ critical
points in $ ]\sigma,a[$. Our induction hypothesis applies to
$ u_{|_{]\sigma,a[}}$ so there is an increasing piecewise $C^1$
function $\tilde v$ on $[\sigma,a]$ with
$\tilde v(\sigma) =0\,,\, \tilde v(a) =1$ such that
$$ 
I_{]\sigma,a[} (\tilde v) \leq I_{]\sigma,a[} (u)\,.
$$
We then get the desired conclusion with
$ v= 1_{]\sigma,a[} \tilde v$, and we have reduced the proof to
the case when $u(s) >0$ for $s>0$.
\par Similarly we get a reduction to the case when $u(s) <1$ for
$s<a$, so we can now assume that $u(s) \in ]0,1[$  for
$0<s<a$ (and that 
  (\ref{red.4.5}) holds).
\par When $s$ increases from $0$ to $ a $, $u$ will first increase
until it reaches a non-degenerate local maximum at some point $s_0\in ]0,a[$
with $u(s_0)\in ]0,1[$, then $u'<0$ on some interval
$]s_0,s_0+\epsilon[$.
Choose $\sigma \in ]s_0,s_0+\epsilon[$ and put $v_1(s)=\min
(u(s),u(\sigma ))$, $0\le s\le \sigma $. Then $v_1\in
C^1_{pw}([0,\sigma ])$, $v_1\ge 0$, $v_1(\sigma )=u(\sigma )$ and
$I_{]0,\sigma [}(v_1)\le I_{]0,\sigma [}(u)$.

\par Clearly $u_{|_{]\sigma,a[}}$ has $M$
critical points  and by the induction assumption (cf.\ Footnote \ref{dil_ext}) we have  a piecewise $C^1$ function $v_2$ on $[\sigma,a]$ with $v'_2 \geq 0$, $v_2(\sigma)=u(\sigma)$, $v_2(a)=1$ s.t. $ I_{]\sigma,a[}(v_2) \leq I_{]\sigma,a[}(u)$.
We get the desired conclusion with $v=1_{[0,\sigma]} v_1 + 1_{]\sigma,a]} v_2$.
 
 \paragraph{Step 2.} 
Let $u\in H_{0,1}^1$ with $u'\ge 0$. Then $u'\in L^2(]0,a[)\subset
L^1(]0,a[)$ has mass 1 and we can find a sequence $ v_j\in C^\infty
([0,a];]0,\infty [)$, $j=1,2,...$ such that
$$
v_j>0,\ \ \int_0^av_jds=1,\ \ v_j\to u' \hbox{ in } L^2.  
$$
If $u_j(s)=\int_0^sv_j(\sigma )d\sigma $, we have $u_j'>0$,
$u_j(0)=0$, $u_j(a)=1$ and $u_j\to u$ uniformly and hence in
$L^2$. Since $u_j'\to u'$ in $L^2$ we have that $u_j\to u$ in
$H_{0,1}^1$. From (\ref{red.4}) we then get 
\begin{equation}\label{red.4.7}
\inf_{u\in H_{0,1}^1}I(u)\ge \inf _{\{u\in H_{0,1}^1\cap C^\infty ([0,a]);\, u'> 0\}} I(u).
\end{equation}

\paragraph{Step 3.}
 Now, let $u\in H_{0,1}^1 \cap C^1([0,a])$ satisfy $u' > 0$  and let us construct $\widetilde{u}\in
{\cal H}$ such that $I(\widetilde{u})\le I(u)$. Let $v\in C^1(]0,a[)$
satisfy
\begin{equation}\label{red.5}
v'^2-v^2=(u'^2-u^2)_+,\ v(0)=0,\ v'\ge 0.
\end{equation}
 We can then apply the global Cauchy-Lipschitz theorem to
$$v'=\sqrt{v^2 +\phi}\,,\, v(0)=0\,,$$ with  $\phi =(u'^2-u^2)_+ \geq
0$. \\ The function $f(x,v):= \sqrt{v^2 +\phi (x)}$ is indeed Lipschitz
in $v$  along the graph of $v$, since $v^2+\phi >0$.  Then $v'\ge v \,  \geq 0$ and we now claim
 that $v\ge u$. 
 From (\ref{red.5}), we get indeed
$$
v'^2-v^2\ge u'^2-u^2,
$$
which can be rewritten as 
$$
v'^2-u'^2\ge v^2-u^2\,.
$$
Factorizing both members in the last estimate and dividing with \break 
$v'+u' \geq u' >0$, we get
\begin{equation}\label{red.6}
(v-u)'\ge \frac{v+u}{v'+u'}(v-u).
\end{equation}
Here $(v+u)/(v'+u')\ge 0$, so the differential inequality
(\ref{red.6})   and $v(0)-u(0)=0$ imply  that
\begin{equation}\label{red.7}
v-u\ge 0.
\end{equation}
In particular, $v(a)\ge u(a)=1$. By (\ref{red.5}) we have
$I(v)=I(u)$. \\
Put $\widetilde{u}=v(a)^{-1}v\in {\cal H}$. Then
$$
I(\widetilde{u})=v(a)^{-2}I(v)=v(a)^{-2}I(u)\le I(u).
$$

\paragraph{End of the proof.}
Putting all the steps together, we can  for any $u\in H^1_{0,1}$,  construct a sequence $\widetilde u_n$ in $C^\infty([0,a])\cap H_{0,1}^1 (]0,a[) $ such that such that $\widetilde u_n '>0$
 on $[0,a]$ and $I(\widetilde u_n) \leq I(u)+ \epsilon_n\,.$ with $\epsilon_n \rightarrow 0$.
Using Step 3  for $\tilde u_n$, we find $\widehat u_n \in \mathcal H$ such that $I(\widehat u_n) \leq I(\widetilde u_n)$.
This completes the proof of the lemma.
\end{proof}

\subsection{Existence of minimizers}\label{exi}
As above, let $a\in ]0,+\infty [$.
We show that the infimum above is attained, i.e.\ that minimizers exist.
\begin{prop}\label{exi1}
  There exists $u_0\in {\cal H}$ such that
  \begin{equation}\label{exi.1}
I_{\mathrm{inf}}(a):= \inf_{u\in {\cal H}}I_{]0,a[}(u)=I_{]0,a[}(u_0).
  \end{equation}
\end{prop}
\begin{proof}
The proof is standard. We recall it for completeness.
Let $\| \cdot \| $ denote the norm in $L^2(]0,a[;m^2 ds)$  and define the norm in $H^1(]0,a[)$ by
$$
\| u\| _1^2 = \| u\| ^2 + \| \partial_s u \| ^2\,.
$$
Under our assumption on $m$, this norm is equivalent to the standard norm (corresponding to $m=1$).
Then 
$$
I_{]0,a[} (u)=\| u\| ^2_1 - 2 \|  u\| ^2\,.
$$
For $u\in \mathcal H$ we have $0\leq u \leq 1$, so $\| u\| ^2\leq
C_ma$, $C_m=\int_0^am^2ds$. Hence, if $u\in \mathcal H$ and  $I_{]0,a[} (u)\leq C$, we have 
$$
\| u\| ^2_1 \leq C + 2 \| u\| ^2 \leq C +2C_ma\,.
$$
A closed ball in $H^1(]0,a[)$ of finite radius is compact for the weak topology in $H^1$. It follows that every set $\{u\in \mathcal H;I_{]0,a[}(u) \leq C\} $ has the same property.\\
Let $u_1,u_2, \dots \in \mathcal H$ be a sequence such that $I_{]0,a[}(u_\nu) \rightarrow \inf_{\mathcal H} I_{]0,a[}$ as $\nu \rightarrow +\infty$.\\
After extracting a subsequence, we may assume that there exists $u_0\in \mathcal H$ such that
$$
u_\nu \rightharpoonup  u_0 \mbox{ in } H^1(]0,a[)\;,\; u_\nu \rightarrow u_0 \mbox{ in } H^{\frac 34} (]0,a[)\,.
$$
We then deduce by continuity of the trace that $u_0(0)=0$, $u_0(a)=1$. Also $0 \leq u_0 \leq u'_0$ in the sense of distributions. Hence $u_0\in \mathcal H$ and consequently 
$$ \inf_{u\in \mathcal H} I_{]0,a[}(u) 
 \leq I_{]0,a[}(u_0).
 $$
Clearly $\| u_\nu\| ^2 \rightarrow \| u_0\| ^2$. From 
$$
\| u_0\| _1 ^2= \lim_{\nu \rightarrow +\infty} (u_0,u_\nu)_1 \leq ||u_0||_1\, \limsup_{\nu \rightarrow +\infty} ||u_\nu||_1\,,
$$
we see that
$$
 \| u_0\| _1 \leq \limsup \|  u_\nu\| _1\,.
$$
Hence 
\begin{multline*}
I_{]0,a[}(u_0) = \| u_0\| ^2_1 -\| u_0\| ^2\\
 \leq \limsup (\| u_\nu\| _1^2 -\| u_\nu\| ^2) =\limsup I_{[0,a|} (u_\nu) =\inf_{u\in \mathcal H} I_{]0,a[}(u)\,.
\end{multline*}
\end{proof}

\par We have the following easy generalization. \\
Let $\sigma ,\tau
,S,T\in \mathbb R$, $\sigma <\tau $, $S,T\geq 0$, $T\ge e^{\tau -\sigma
}S$. Let $0<m\in C^1([\sigma ,\tau ])$ and define 
\begin{equation}\label{exi.3}
{\cal H}_{\sigma ,\tau }^{S,T}=\{ u\in H^1(]\sigma ,\tau [;\mathbb R);\,
u(\sigma )=S,\ u(\tau )=T,\ 0<u\le u'\}
\end{equation}
as in (\ref{red.3}).\\
We then wish to study
$$
\inf _{u\in {\cal H}_{\sigma ,\tau }^{S,T}}I_{]\sigma ,\tau [}(u),
$$
where
$$
I_{]\sigma ,\tau [}(u)=\int_\sigma ^\tau (u'^2-u^2)m^2ds.
$$
The preceding proposition has a straight forward generalization:
\begin{prop}\label{exi3}  There exists 
  $u_0\in {\cal H}_{\sigma ,\tau }^{S,T}$ such that
  \begin{equation}\label{exi.4}
\inf _{u\in {\cal H}_{\sigma ,\tau }^{S,T}}I_{]\sigma ,\tau [}(u)=I_{]\sigma ,\tau [}(u_0).
  \end{equation}
\end{prop}
In the situation of the last proposition we call $u_0$ a minimizer in
${\cal H}_{\sigma ,\tau }^{S,T}$.
\subsection{On $m$-harmonic functions}\label{mh}
\subsubsection{Minimizers and   $m$-harmonic functions} By \eqref{eq:defK}, we have $K_m= m^{-2} P_m$ where 
$$
P_m=-\partial _s\circ m^2\circ \partial _s -m^2. 
$$
If $0\le \sigma <\tau <+\infty $, we say that a function $u$ on $]\sigma ,\tau [$ is $m$-harmonic if
$P_mu=0$ on that interval.

\par  The operator  $K_m$ is an unbounded self-adjoint operator in
$L^2(]\sigma ,\tau [;m^2ds)$ when equipped with the domain ${\cal
  D}=(H_{0,0}^1\cap H^2)(]\sigma ,\tau [)$. 
  It has discrete spectrum,
contained in some interval $[-C,+\infty [$. If $\tau
  \le a$ for some fixed $a\in ]0,+\infty [$ and if $\tau -\sigma $ is
small enough\footnote{ \label{f1} More precisely,
  there exist $C,\epsilon _0>0$ such that, for $|\sigma-\tau| <
  \epsilon_0$, the  Dirichlet realization in $]\sigma,\tau[$ 
    (also denoted by $P_m$) satisfies the lower bound 
$m^{-2}P_m\ge \frac 1C$.}
we have
\begin{equation}\label{mh.1}
m^{-2}P_m\ge 1/{|\mathcal O (1) |}.\end{equation}

Then $P_m:H_{0,0}^1\cap H^2\to H_0$ is a bijection and it is straight
forward to see that for all $S,T\in \mathbb R$, the problem
\begin{equation}\label{mh.2}
  \begin{cases}
    P_mu=0\hbox{ on }]\sigma ,\tau [,\\
    u(\sigma )=S,\ u(\tau )=T,
  \end{cases}
\end{equation}
has a unique solution $u\in H^2(]\sigma ,\tau [)$. 
\par Indeed, let $f\in
C^2([\sigma ,\tau ])$ satisfy $f(\sigma )=S$, $f(\tau )=T$ and put
$u=f+\widetilde{u}$, where $\widetilde{u}\in H_{0,0}^1\cap H^2$ is the
unique solution in $H_{0,0}^1\cap H^2$ of $P_m\widetilde{u}=-P_mf$. We
denote by $u=f_{\sigma ,\tau }^{S,T}$ the unique solution of
(\ref{mh.2}).

\par The property (\ref{mh.1}) is equivalent to
\begin{equation}\label{mh.3}
I_{]\sigma ,\tau [}(u)\ge \frac{1}{C}\| u\|_{H^1}^2,\ \forall u\in
H_{0,0}^1(]\sigma ,\tau [).
\end{equation}
 Recall the definition of
$
H^1_{S,T}(]\sigma ,\tau [)$ in \ref{red.0.5}.
A general element $u\in H_{S,T}^1$ can be written
\begin{equation}\label{mh.4}
u=f+\widetilde{u},\ \widetilde{u}\in H_{0,0}^1(]\sigma ,\tau [),
\end{equation}
where $f=f^{S,T}_{\sigma ,\tau }$. \\
We have with $I=I_{]\sigma ,\tau
  [}$:
\begin{equation}\label{mh.5}
  \begin{split}
I(u)&=I(\widetilde{u})+2\int_\sigma ^\tau
(f'\widetilde{u}'-f\widetilde{u})m^2ds+I(f)\\
&\ge
\frac{1}{C}\|\widetilde{u}\|^2_{H^1}-C_{S,T}\|\widetilde{u}\|_{H^1}-C_{S,T}\\
&\ge \frac{1}{2C}\|\widetilde{u}\|_{H^1}^2-\widetilde{C}_{S,T}\,.
  \end{split}
\end{equation}
Thus $$\|\widetilde{u}\|_{H^1}^2\le \mathcal O (1) (I(u)+1)\,,$$
 and combining this
with the estimate $$\|u\|_{H^1}^2\le 2(\|\widetilde{u}\|_{H^1}^2+\|
f\|_{H^1}^2 ),$$ we get
\begin{equation}\label{mh.6}
\|u\|_{H^1}^2\le C_{S,T}(I(u)+1),
\end{equation}
with a new constant $C_{S,T}$.
\begin{prop}\label{mh1}
Let $\sigma ,\tau $ satisfy (\ref{mh.1}) or equivalently (\ref{mh.3})
and fix $S,T\in \mathbb R$. Then there exists $ u_0\in H^1_{S,T}(]\sigma ,\tau
[)$ such that
\begin{equation}\label{mh.7}
I_{]\sigma ,\tau [}(u_0)=\inf_{u\in H^1_{S,T}(]\sigma ,\tau
  [)}I_{]\sigma ,\tau [}(u).
\end{equation}
The minimizer $u_0$ is equal to the unique solution $f_{\sigma ,\tau }^{S,T}$ of
(\ref{mh.2}) and hence belongs to $H^2(]\sigma ,\tau [)$.
\end{prop}

\begin{proof}
Thanks to (\ref{mh.6}) we can adapt the proof of Proposition \ref{exi3}
to see that there exists  $u_0\in H^1_{S,T}(]\sigma ,\tau [)$, satisfying
(\ref{mh.7}). The standard variational argument then shows that
$u_0=f^{S,T}_{\sigma ,\tau }$ solves (\ref{mh.2}) and is therefore the
unique minimizer. $P_m$ being elliptic,  we have $u_0\in H^2(]\sigma ,\tau [)$.
\end{proof}

\begin{remark} \label{mh2}
  Let $0\le \sigma <\tau $, $0\le S<T$, with $T\ge e^{\tau -\sigma }S$
  as in \eqref{condnecST} and assume that (\ref{mh.1})
    holds on $]\sigma ,\tau [$. If $u_0:=f^{S,T}_{\sigma ,\tau }$
  belongs to ${\cal H}_{\sigma ,\tau }^{S,T}$ (i.e.\ if
  $0\le u_0\le u_0'$), then $u_0$ is equal to the unique minimizer in
  $H_{S,T}^1(]\sigma ,\tau [)$ and hence it is a minimizer in the
  smaller space ${\cal H}_{\sigma ,\tau }^{S,T}$. If $u_1$ is another
  minimizer in that space, then $I(u_1)=I(u_0)$, so it is also a
  minimizer in $H_{S,T}^1$ and by the uniqueness in that space,
  $u_1=u_0$.
\end{remark} 

\begin{remark}\label{mh3}
  Let $u_0$ be a minimizer in ${\cal H}^{S,T}_{\sigma ,\tau }$, let
  $\sigma \le \widetilde{\sigma }<\widetilde{\tau }\le \tau $ and set
  $\widetilde{S}=u_0(\widetilde{\sigma })$,
  $\widetilde{T}=u_0(\widetilde{\tau })$. Then
  ${{u_0}_\vert}_{]\widetilde{\sigma },\widetilde{\tau }[}$ is a
  minimizer in
  ${\cal H}^{\widetilde{S},\widetilde{T}}_{\widetilde{\sigma
    },\widetilde{\tau }}$. If
  $f^{\widetilde{S},\widetilde{T}}_{\widetilde{\sigma
    },\widetilde{\tau }}$ belongs to
  ${\cal H}^{\widetilde{S},\widetilde{T}}_{\widetilde{\sigma
    },\widetilde{\tau }}$ (assuming (\ref{mh.1}) holds on $]\widetilde{\sigma
    },\widetilde{\tau }[$), then
  ${{u_0}_\vert}_{]\widetilde{\sigma },\widetilde{\tau
    }[}=f^{\widetilde{S},\widetilde{T}}_{\widetilde{\sigma
    },\widetilde{\tau }}$.
\end{remark}

\subsubsection{Riccati equations and  $m$-harmonic functions.}\label{sss3.4.2}
\par We next discuss $m$-harmonic functions from the point of view of
first order non-linear ODE's,  more specifically Riccati equations. Let $f$ be an $m$-harmonic function on $]\sigma ,\tau [$  such
that
\begin{equation}\label{mh.8}0<f\le f'.\end{equation}
(For some arguments we relax this condition somewhat,
  still assuming that $f,f'>0$.) 
Put $$\mu =m'/m\,.$$
 Then from $$(\partial _s\circ m^2\circ \partial
_s+m^2)f=0\,,$$
we get
\begin{equation}\label{mh.9}
(\partial _s^2+2\mu \partial _s+1)f=0\,.
\end{equation}
Writing $$\phi =\log  f \mbox{ and }  \psi =\phi '=f'/f\,,$$
we get
\begin{equation}\label{mh.10}
 \psi \ge 1 \mbox{ and } \phi ''+\phi '^2+2\mu \phi '+1=0\,.
\end{equation}
We can rewrite the last equation in one of the two equivalent forms
  \begin{equation}\label{mh.11}
    \psi '=-(\psi ^2+2\mu \psi +1)\ 
   \mbox{ or } 
\psi '=-2\left( \mu +\frac{1}{2}\left( \psi +\frac{1}{\psi } \right)
\right)\psi\,.
\end{equation} In the region $\psi >1$ we can determine more explicitly when
we have $\psi '>0\,$, i.e.\ when $$\psi ^2+2\mu \psi +1<0\,,$$ or equivalently
when
$$
\mu <-\frac{1}{2}\left(\psi +\frac{1}{\psi } \right).
$$
Here the right hand side is $\le -1$, so we have the necessary condition
that $$ \mu <-1\,.$$
 Assuming this to hold, we notice that $\psi ^2+2\mu
\psi +1$ vanishes precisely for $\psi =-\mu \pm
\sqrt{\mu ^2-1}$. Clearly, $-\mu +\sqrt{\mu ^2-1}>1$ when $\mu <-1$. A small
calculation (or using that the product of the two solutions is equal
to 1) shows that $-\mu -\sqrt{\mu ^2-1}<1$ when $\mu < -1$.\\ 

In conclusion, we have proven
\begin{lemma}\label{mh4}
  Consider a point $s$ where (\ref{mh.11}) holds and $\psi(s) >1$. Then:
  $$
\psi ' (s)>0\hbox{ if and only if  }\mu (s)  <-1\hbox{ and }1<\psi(s) <-\mu(s) +\sqrt{\mu ^2(s)-1}.
  $$
\end{lemma}
\par
  We now put
  $$
  f_+(s)=\begin{cases}
    1,\hbox{ if }\mu (s)\ge -1,\\
    -\mu (s)+\sqrt{\mu (s)^2-1},\hbox{ if }\mu (s)<-1\,.
  \end{cases}
  $$
  The last lemma tells us that
  \begin{equation}\label{mh.13}\psi '(s)\le 0,\hbox{ when (\ref{mh.11})  holds and } \psi (s)\ge f_+(s)\,.\end{equation}
This implies the following  nice control of solutions of (\ref{mh.11}) in the
direction of increasing  ``time'' $s$:\\
 If (\ref{mh.11}) holds for $\sigma <s<\tau $
and $\sigma <s_0<\tau $, then
\begin{equation}\label{mh.14}\psi (s)\le \max (\psi
  (s_0),\max_{[s_0,s]}f_+),\ s_0\le s<\tau\, .\end{equation}

\par Let $a\in ]0,+\infty [$ be fixed and assume that $0\le \sigma <\tau \le a$ with $\tau -\sigma $ small
enough, so that the Dirichlet realization of  $m^{-2}P_m$ is $\ge 1/|{\cal O}(1)|$ and let
$f=f^{S,T}_{\sigma ,\tau }$, so that $u=f$ satisfies
(\ref{mh.2}). We restrict the attention to a region
  $ \{ s \in ]0,a[\,;\, \psi (s)\in ]1/2,2C_0[\}$ where $C_0$ can be large but fixed. We
  have
\begin{equation}\label{mh.15}
\int_\sigma ^\tau \psi (s)ds=\int_\sigma ^\tau \partial _s\log  f ds=\log 
\frac{f(\tau )}{f(s)}=\log  \frac{T}{S}.
\end{equation}

\par Conversely, from (\ref{mh.15}), (\ref{mh.11}), we get $f(\tau
)/f(\sigma )=T/S$ and after multiplying $f$ with a suitable positive
constant, we get (\ref{mh.2}).

\par Consider   the differential equation (\ref{mh.11})  over an interval
$]\sigma ,\tau [$ with $1/2\le \psi \le 2C_0$ with $C_0 >1$ as above.
 If $\tau -\sigma $ is
small enough, we have a unique such solution if we prescribe
$\psi(\sigma)$ in the slightly smaller interval $]2/3,
  3 C_0/2[$ and  we get $\psi (s)=\psi (\sigma )+{\cal O}(s-\sigma )$. 
Hence we have 
\begin{equation}\label{mh.16}
m_{]\sigma ,\tau [}(\psi ) := \frac{1}{\tau -\sigma }\int_\sigma ^\tau \psi (s)ds=\psi (\sigma
)+{\cal O}(\tau -\sigma )\,.
\end{equation}
\par 
For $z\in ]2/3, 3 C_0/2[$, we define $\widetilde{m}_{\sigma,\tau}
(z):=m_{]\sigma ,\tau [}(\psi )$ where $\psi$ is the solution of (\ref{mh.11}) with $\psi(\sigma) = z$.  
$\widetilde{m}_{\sigma ,\tau }$ can be extended to a biholomorphic map from some fixed neighborhood of $ [1,C_0]$ in $\mathbb C$ onto a $(\sigma ,\tau  )$-dependent neighborhood of the same
type and \ref{mh.16} extends to the estimate:
$$
\widetilde{m}_{\sigma ,\tau }(z)=z+{\cal O}(\tau -\sigma ).
$$
The inverse map $\widetilde{m} \mapsto z$ satisfies trivially
$$
z=\widetilde{m}_{\sigma ,\tau }(z )+{\cal O}(\tau -\sigma ),
$$
and this holds uniformly for $\sigma ,\tau \in [0,a]$, $|\tau -\sigma
|\ll 1$.  (Once $z$ has been determined from some $\widetilde m$, 
we determine $\psi $  from the differential
equation (\ref{mh.11}) with initial condition $\psi(\sigma) =z$ and we have 
$\widetilde{m} =m_{]\sigma
  ,\tau [}(\psi )$\,.)

We can apply this to (\ref{mh.15}), that we write as
$$
\widetilde{m}_{\sigma ,\tau }(z)=m_{]\sigma ,\tau [ }(\psi )=\frac{1}{\tau -\sigma }\log  \frac{T}{S}.
$$
If $(\tau -\sigma )^{-1}\log  (T/S)\in
\mathrm{neigh\,}([1,C_0],\mathbb R)$,  we get\footnote{Here we use
  "neigh($A,B)$" as an abbreviation for "some neighborhood of $A$ in
  $B$".} a unique $z$ and a real solution $\psi $ of (\ref{mh.15}), 
(\ref{mh.11}) with
$\psi (\sigma )=z$,
\begin{equation}\label{mh.17}
\psi (s)=\frac{1}{\tau -\sigma }\log  \frac{T}{S}+{\cal O}(\tau -\sigma ),
\end{equation}
uniformly on $[\sigma ,\tau ]$. In particular, if
\begin{equation}\label{mh.18}
\frac{1}{\tau -\sigma }\log  \frac{T}{S}\ge 1+ \frac{1}{| \mathcal O(1)|},
\end{equation}
we get
\begin{equation}\label{mh.19}
\psi (s)\ge 1+  \frac{1}{|\mathcal O(1)|}-{\cal O}(\tau -\sigma )>1
\end{equation}
and we conclude that the corresponding solution
  $u=f_{\sigma ,\tau }^{S,T}$ belongs to ${\cal H}_{\sigma ,\tau
  }^{S,T}$.

  In conclusion:
\begin{prop}\label{mh5}
For every $C_0>1$, there exist  $ \epsilon_0>0$ and $C_1>0$ such that if $S,T>0\,$,
$0\le \sigma <\tau \le a\,$, $\tau -\sigma <\epsilon _0\,$, $2/3\le \ln
(T/S)/(\tau -\sigma )\le 3C_0/2\,$, then $f=f_{\sigma ,\tau }^{S,T}$
satisfies
$$\left| \frac{f'}{f}-\frac{\ln (T/S)}{\tau -\sigma } \right|\le
C_1(\tau -\sigma )\hbox{ on }]\sigma ,\tau [\,.$$
 In particular, if
$$
\frac{\ln (T/S)}{\tau -\sigma }\in [1+1/C_2,3C_0/2],
$$
where $C_2>0$, then
$$
\frac{f'}{f}-1\ge \frac{1}{C_2}-C_1(\tau -\sigma )\hbox{ on }]\sigma
,\tau [,
$$
hence $f'/f-1\ge 1/(2C_2)$ on $]\sigma ,\tau [$ and $f\in {\cal
  H}_{\sigma ,\tau }^{S,T}$, if $\tau -\sigma $ is small enough.
\end{prop}

\subsection{Structure of minimizers}\label{stm}
We will first discuss minimizers over a fixed interval $]0,a[$\,,
$0<a<\infty $.
It may be useful to recall that if $u\in H^1(]0,a[)$, then $u$ is H\"older continuous of
order $1/2$, i.e.  $u\in C^{1/2}$. In fact, if $0\le \sigma <\tau \le a\, $,
$$
|u(\tau )-u(\sigma )|\le \int _\sigma ^\tau | u'(s) |\, ds\le \|
u\|_{H^1}(\tau -\sigma )^{1/2}.
$$
\begin{prop}\label{stm1}
Let $\phi \in H^1(]0,a[)$ be real-valued, $0\le \sigma <\tau \le
a$. Let
$$
\lambda =\frac{\phi (\tau )-\phi (\sigma )}{\tau -\sigma }.
$$
 Then there exist arbitrarily short intervals $[\widetilde{\sigma
 },\widetilde{\tau } ]\subset [\sigma ,\tau ]$ such that
 $$
\frac{\phi (\widetilde{\tau })-\phi (\widetilde{\sigma
    })}{\widetilde{\tau }-\widetilde{\sigma
  }}=\lambda \, .
 $$
\end{prop}
\begin{proof}~\\
  If $I=[\widetilde{\sigma },\widetilde{\tau }]$ is a subinterval of
  $[\sigma ,\tau ]$, then 
  $m_I(\phi' ):=(\phi (\widetilde{\tau })-\phi (\widetilde{\sigma
  }))/(\widetilde{\tau }-\widetilde{\sigma })$ is equal to the average
 over $I$ of $\phi' $. Let $N\in \mathbb N$,
  $N\ge 2$, and decompose $[\sigma ,\tau ]$ into the disjoint union of
  $N$ intervals $I_1,...,I_N$ of length $(\tau -\sigma )/N$. Then the
  mean value of the averages $m_{I_j}(\phi' )$ is equal to
  $\lambda $. If no such average is equal to $\lambda $, there exist
  $I_j$, $I_k$, with $m_{I_j}(\phi' )<\lambda $,
  $m_{I_k}(\phi' )>\lambda $. Let $I^t=I_j+Ct$ with $C\in  \mathbb R $ chosen so that $I^0=I_j$, $I^1=I_k$ or at least so that we
  have equality for the interiors. Then $m_{I^t}(\phi' )$
  varies continuously with $t$, so there exists  $ t\in ]0,1[$ such that
  $m_{I^t}(\phi' )=\lambda $. 
\end{proof}

\par Let $u_0\in {\cal H}(]0,a[)={\cal H}_{0,a}^{0,1}$ be a
minimizer for $I_{]0,a[}$. Let $$\psi _0=u_0'/u_0=\phi _0' \mbox{  where } \phi _0=\log 
u_0\,.$$
Then, we deduce from \eqref{exi.3}:
\begin{equation}\label{stm.1}
\psi _0\ge 1\,,
\end{equation}
\begin{equation}\label{stm.2}
  {{\phi _0}_\vert}_{]\epsilon ,a[}\in H^1(]\epsilon ,a[)\,,
\end{equation}
for every $\epsilon >0\,$.
\begin{prop}\label{stm2}
Let $0<\sigma <\tau \le a$ and assume that $\lambda :=m_{]\sigma ,\tau [}(\psi
_0)>1$. Then there exists $ s_0\in ]\sigma ,\tau [$ and $\alpha ,\beta $
with
\begin{equation}\label{stm.3}
0\le \alpha <s_0<\beta \le a,
\end{equation}
such that
 \begin{equation}\label{stm.4}
u_0 \hbox{ is }m\hbox{-harmonic and } 0<u_0<u_0'\hbox{  on }]\alpha ,\beta [\,.
\end{equation}
\begin{itemize}
\item
 If $\beta <a$, we have $\psi _0(s)\to 1$, $s\nearrow \beta $.
 \item If $\alpha >0$, we have $\psi _0(s)\to 1$, $s\searrow \alpha  $.
\end{itemize}
We recall that $\psi _0(s)\to -\infty $, when $s\searrow 0$. Moreover 
$s_0$ can be chosen so that $\psi _0(s_0)$ is arbitrarily close to
$\lambda $.
\end{prop}
\begin{proof}
By Proposition \ref{stm1} there exist arbitrarily short intervals
$I=]\widetilde{\sigma },\widetilde{\tau }[\subset ]\sigma ,\tau [$
such that $m_I(\phi' _0)=\lambda $. For each such interval put
$S=u_0(\widetilde{\sigma })$, $T=u_0(\widetilde{\tau })$, so that $\log 
(T/S)=\lambda (\widetilde{\tau }-\widetilde{\sigma })$. Let
$f=f_{\widetilde{\sigma },\widetilde{\tau }}^{S,T}$, $\psi =f'/f$,
$\phi =\log  f$. Then we have (\ref{mh.11}), (\ref{mh.15}) and we can
apply (\ref{mh.17}) (or Proposition \ref{mh5}) with $\sigma ,\tau $ there replaced by
$\widetilde{\sigma },\widetilde{\tau }$, to see that $\psi (s)=\lambda
+{\cal O}(\widetilde{\tau }-\widetilde{\sigma })$, $\widetilde{\sigma
}\le s\le \widetilde{\tau }$. In particular $\psi >1$ on
$[\widetilde{\sigma },\widetilde{\tau }]$ when $\widetilde{\tau
}-\widetilde{\sigma }$ is small enough. Hence $f\in {\cal
  H}_{\widetilde{\sigma },\widetilde{\tau }}^{S,T}$ and applying
Remark \ref{mh3}, we conclude that
\begin{equation}\label{stm.5}
u_0=f_{\widetilde{\sigma },\widetilde{\tau }}^{u_0(\widetilde{\sigma
  }), u_0(\widetilde{\tau })}\hbox{ on }]\widetilde{\sigma }
  ,\widetilde{\tau }[.
\end{equation}
Choose $s_0\in ]\widetilde{\sigma },\widetilde{\tau }[$ so that $\psi
(s_0)=\psi _0(s_0)$ is as close to $\lambda $ as we like.

\par Let $]\alpha ,\beta [\subset ]0,a[$ be the largest open interval
containing $s_0$ on which $u_0$ is $m$-harmonic and $u_0'/u_0>1$.

\par Assume that $\alpha >0$ and that $\psi _0(\alpha +0)>1$. Then we
can find (new) arbitrarily short intervals $]\widetilde{\sigma
},\widetilde{\tau }[$, containing $\alpha $, such that
$$
m_{]\widetilde{\sigma },\widetilde{\tau }[}(\phi' _0)\ge
1+{ \frac{1}{|{\cal O}(1)|}}
$$
and as above, we see that $u_0$ is $m$-harmonic and $u_0'/u_0>1$ on
$]\widetilde{\sigma },\widetilde{\tau }[$ which contradicts the
maximality of $]\alpha ,\beta [$. Hence, if $\alpha >0$, we have $\psi
_0(\alpha +0)=1$.

\par Similarly, if $\beta <a$, we have $\psi _0(\beta -0)=1$.
\end{proof}

\par Let $J\subset ]0,a[$ be the countable disjoint union of all
open maximal intervals $I\subset ]0,a[$, such that $u_0$ is
$m$-harmonic with $u_0'/u_0>0$ on $I$.
\begin{prop}\label{stm4}
$\psi _0$ is uniformly Lipschitz continous on $]0,a[\,$, $>1$ on $J$, and $=1$ on
$]0,a[\setminus J$.
\end{prop}

\begin{proof}
For $t\in ]0,a[$\,,  let $\partial_s^+\psi_0 (t)$ be the set of all limits $(\psi_0(t+\epsilon_j) - \psi_0(t))/\epsilon_j$ with $\epsilon_j \searrow 0$.
Similarly  let $\partial_s^-\psi_0 (t)$ be the set of all limits $(\psi_0(t+\epsilon_j) - \psi_0(t))/\epsilon_j$ with $\epsilon_j \nearrow 0$. We can also define 
$\partial_s^-\psi_0(a)$.\\
When $t\in \mathcal J$, we have 
$$
\partial_s^+\psi_0(t)=\partial_s^-\psi_0(t) =\{\psi_0'(t)\}\,.
$$
When $t\in ]0,a[ \setminus \mathcal J$, we  see using \eqref{mh.11} that
\begin{equation}\label{eq:3.41a}
\partial_s^+ \psi_0 (t) \subset \left \{
 \begin{array}{ll} 
~ \{ 0 \} & \mbox{ if } \mu \geq -1 \\
 ~[0 \,,\,-2 (1+\mu)] & \mbox{ if } \mu  < -1 
\end{array} 
\right. \,,
\end{equation}
\begin{equation}\label{eq:3.42a}
\partial_s^- \psi_0 (t) \subset \left \{
 \begin{array}{ll} 
   [-2 (1+\mu)\,,\, 0]  & \mbox{ if } \mu  \geq  -1 \\
 \{ 0 \} & \mbox{ if } \mu < -1 \\
\end{array} 
\right. \,.
\end{equation}
From this it follows that $\psi_0$ is Lipschitz.
 \end{proof}

We next discuss some consequences for the global
structure of minimizers. As before, let $u_0\in {\cal H}$ be a
minimizer for $I_{]0,a[}$ and recall that $u_0$ is $m$-harmonic
with $u_0'/u_0>1$ on a
countable union $J$ of maximal open subintervals of $]0,a[$. One of
these subintervals is of the form $]0,\widetilde{a}[$, for some
$\widetilde{a}\in ]0,a]$, which is uniquely determined while
$u_0\,_{\big | ]0,\widetilde a[}$ is unique up to a  positive constant factor.\\ 
We have
$\widetilde{a}=a$  if $\mu \le -1$ on $]0,a[$.  In
  fact, by (\ref{mh.11}),
  \[
    \begin{split}
      (\psi -1)'&=-2(1+\mu )-2(1+\mu )(\psi -1)-(\psi -1)^2\\
      &\ge -2(1+\mu )(\psi -1)-(\psi -1)^2,
    \end{split}
  \]
so we cannot reach the region $\psi -1=0$ in finite positive time from a point in
the region $\psi -1>0$.

\par When $\widetilde{a}<a$, if $\mu (s)\ge -1$ for
$\tilde a\le s\le a$ (and in particular if $\mu (s)\ge -1$ on $]0,a[$),
it follows from Lemma \ref{mh4} that $\psi _0(s)=1$ for $s\ge
\widetilde a$. Indeed, otherwise there would be a maximal open subinterval
$]b,c[\subset ]\widetilde a,a[$ on which $\psi _0$ is $m$-harmonic, in
contradiction with the fact that $\psi '_0\le 0$ there.

\par More generally, let $\widetilde{a} <a$ and assume that $\psi
_0\not\equiv 1$ on $]\widetilde a,a[$. Let $I=]\sigma ,\tau [$ be a maximal
open subinterval of $ ]\widetilde a,a[$ on which $u_0$ is $m$-harmonic
with $u_0'/u_0>1$. Then
$\psi _0>1$ on $]\sigma ,\tau [$ and converges to $1$ when $s\searrow
\sigma $. When $\tau <a$ we also have that $\psi _0\to 1$ when
$\sigma \nearrow \tau $. Lemma \ref{mh4} then tells us that there
exist points $s>\sigma $ arbitrarily close to $\sigma $ where $\mu
(s)<-1$. Similarly, if $\tau <a$ there are points $s<\tau $
arbitrarily close to $\tau $ with $\mu (s)>-1$.

\par We get the following conclusion, where we represent $J$
as a disjoint union of maximal subintervals $I$, where $u_0$ is
$m$-harmonic with $u_0'/u_0>1$:
\begin{itemize}
\item[] If $\mu \ge -1$ on $]0,a[$, then $J=I=]0,\widetilde{a}[$ for some
  $0<\widetilde{a}\le a$.
\item[] If $I=]0,\widetilde a[$, $\widetilde a<a$, then $I$ contains a point
  $s  $ arbitrarily close to $\tilde a$ where $\mu (s)> -1$.
  \item[] If $I=]\sigma ,\tau [$, $\widetilde a<\sigma<\tau <a $, then $I$
    contains two points $\tilde{\sigma }$, $\widetilde{\tau }$,
    arbitrarily close to $\sigma $ and $\tau $ respectively, such that
    $\mu (\widetilde{\sigma })<-1$, $\mu (\widetilde{\tau })>-1$.
    \item[]If $I=]\sigma ,a [$, $\tilde a<\sigma <a $, then $I$
    contains a point $\widetilde{\sigma }$, 
    arbitrarily close to $\sigma $, such that
    $\mu (\widetilde{\sigma })<-1$.
  \end{itemize}

  \par We spell out the conclusion  when $\mu \ge -1$:
  \begin{prop}\label{stm5}
 Let $u_0$ be a minimizer for
$I_{]0,a[}$ on ${\cal H}={\cal H}_{0,a}^{0,1}$ and let
$\widetilde{a}$ be the largest number in $]0,a]$ such that
$u_0$ is $m$-harmonic with $u_0'/u_0>1$ on
  $]0,\widetilde{a}[\,$. If $\widetilde{a}<a$ and $\mu (s)\ge -1$ on
  $[\widetilde{a},a]$, then
 $u_0(s)=e^{s-a}$ on $[\widetilde{a},a[$ and
  $u_0$ is uniquely determined.
\end{prop}

  \begin{remark}\label{stm6}~
  \begin{itemize}
  \item When $\widetilde{a}<a$, we shall see that $\widetilde{a}=a^{*} $
is independent of $a$. See Proposition~\ref{stm8}.
  \item 
The proposition can be applied in the case $m$ constant ($\mu=0$)  and more generally
  the case $m_\alpha (s) =\exp - \alpha s$ with  $\alpha\le 1$.
  \end{itemize}
\end{remark}

\par We end this subsection by studying global minimizers,
more precisely minimizers defined on all of $]0,+\infty [$. Let
$$
{\cal H}(]0,+\infty [):=\{u\in H^1_{\mathrm{loc}}([0,+\infty [);\,
0\le u\le u' ,\ u(0)=0, u>0\hbox{ on }]0,+\infty [\}.
$$
We say that $u_0\in {\cal H}(]0,+\infty [)$ is a minimizer (or a
global minimizer when emphasizing that we work on the whole half axis) if ${{u_0}_\vert}_{]0,a[}$ is a minimizer in ${\cal
  H}_{0,a}^{0,u_0(a)}$ for every $a>0$. Recall our assumption
  that $0<m\in C^1([0,+\infty [)$.
\begin{prop}\label{stm7}
A global minimizer $u_0\in {\cal H}(]0,+\infty [)$ exists.
\end{prop}
\begin{proof}
Let $0<a_1<a_2<...$ be a sequence such that $a_j\to +\infty $ when
$j\to +\infty $. It suffices to find $u_0\in {\cal H}(]0,+\infty [)$
such that ${{u_0}_\vert}_{]0,a_j[}$ is a minimizer in ${\cal
  H}_{0,a_j}^{0,u_0(a_j)}$ for every $j$.

\par Let $u_1\in {\cal H}_{0,a_1}^{0,1}$ be a minimizer (and here we
could replace $1$ by any positive number). Let
$\widetilde{u}_2\in {\cal H}_{0,a_2}^{0,1}$ be a minimizer. Replacing
$\widetilde{u}_2$ by $\widetilde{u}_2(a_1)^{-1}\widetilde{u}_2$, we
get a new minimizer
$\widetilde{u}_2\in {\cal H}^{0,\widetilde{u}_2(a_2)}_{0,a_2}$ with
$\widetilde{u}_2(a_1)=u_1(a_1)\, (=1)$. Then both $u_1$ and
${{\widetilde{u_2}}}_{\vert]0,a_1[}$ are minimizers in ${\cal
  H}_{0,a_1}^{0,1}$, so
$$
u_2:=1_{]0,a_1]}u_1+1_{]a_1,a_2[}\widetilde{u}_2
$$
is also a minimizer in ${\cal H}_{0,a_2}^{0,u_2(a_2)}$ and has the
property: ${{u_2}_\vert}_{]0,a_1[}=u_1$.

\par Iterating this argument, we get a sequence of minimizers $u_j$ in
${\cal H}_{0,a_j}^{0,u(a_j)}$, $j=1,2,...$ such that
${{u_{j+1}}_\vert}_{]0,a_j[}=u_j$ for $j=1,2,...$ and it suffices to
define $u_0$ on $]0,+\infty [$ by ${{u_0}_\vert}_{]0,a_j[=u_j}$.
\end{proof}

\par The discussion of the structure of minimizers in ${\cal
  H}_{0,a}^{0,1}$ applies directly to global
minimizers. In particular, we get:
{ \begin{proposition}\label{stm8}
 If $u_0$ is a global minimizer, then $u_0$
is $m$-harmonic with $u_0'/u_0>1$ on a maximal interval of the form
$]0,a^{*} [$, for some $a^{*} \in ]0,+\infty ]$.  $a^{*} $ is uniquely determined and (the $m$-harmonic
function) ${{u_0}_\vert}_{]0,a^{*} [}$ is unique up to a constant
positive factor.

\par This characterization of $a^*$ is equivalent to the one in (\ref{defa*})
\end{proposition}
\begin{remark}\label{stm9}
Note that we do not claim that we 
have uniqueness for $u_0$ up to multiplication with positive
constants.  However, we do get this uniqueness if we add the assumption
that $\mu (s)\ge -1$ for $s\ge a^{*} $. Cf.\ Proposition \ref{stm5}.
\end{remark}

\par From the discussion with Riccati equations, we have also 
\begin{proposition}
\ If \eqref{eq:C1weak} holds, then $a^{*} $. 
\end{proposition}
\begin{proof} 
Let $A>0$ be such that $\mu(s) \geq -1 + \frac{1}{A}$ for $s \geq A$. It follows from \eqref{eq:3.41a}-\eqref{eq:3.42a} that there exists $B \geq A$ such that $\psi_0 (A) \leq B$.\\
Then we get from \eqref{mh.11} 
$$
\left \{\begin{array}{l} s \geq A \\ \psi(s) >1\end{array} \right. \mbox{ implies } \psi'_0 \leq - 2/A\,,
$$
so 
$\psi_0(s)=1$ for $B - \frac 2A (s-A) \leq 1$, i.e. for $s \geq \frac A2 (B+1)$.

\end{proof}
}

 \subsection{Application to our minimization problem}\label{ss3.6}
Let $u_0:[0,+\infty [\to \mathbb R$, satisfy $P_mu_0=0$, $u_0(0)=0$,
$u_0'(0)>0$, so that $u_0$ is uniqueley determined up to a constant
positive factor. Then $u_0>0$ on $]0,a^*(m)[$ and when $a^*(m)<+\infty
$, we have $u_0'(a^*(m))=u_0(a^*(m))$ and $u_0$ is then the first
eigenfunction of $K_{m,a^*}^{DR}$ with eigenvalue $0$.
 \begin{prop}\label{prop3.18}  For  $a\in ]0,+\infty [\cap ]0, a^{*} (m)]$, 
   \begin{equation} \label{energyabetter}
    I_{\mathrm{inf}}(a)
  = \psi_0(a) m^2(a)\,,
 \end{equation} 
  where $\psi_0= u'_0/u_0$, $u_0$.\\
  In particular, when $a=a^{*} (m)<+\infty $, we get
  \begin{equation} \label{energya**} 
I_{\mathrm{inf}} (a^{*} (m))  = m^2(a^{*} (m))\,.
 \end{equation}
 \end{prop}
 \begin{proof}
   We have seen in Proposition \ref{red1} that
   $$\inf_{\{u\in H^1(]0,a[), u(0)=0, u(a)=1\}} \int _0^a (u'^2-u^2)_+
   \,m^2\, ds = \inf_{u \in \mathcal H }\, \int _0^a (u'^2-u^2) \,
   m^2\, ds\,. $$ Here the minimizer is  $u=u_0(s)/u_0(a)$. Integration by parts and using
   that $u$ is $m$-harmonic, gives
$$
 \int
_0^a (u'^2-u^2) \, m^2\, ds = m^2(a)  u(a) u'(a) = m^2(a) \psi_0 (a)\,.
 $$
 We also recall that $\psi_0 (a^{*} ) =1$.
 \end{proof}
 \begin{remark}
 When $m=1$, we obtain $a^{*}  =\frac \pi 4$ and $\psi_0 (s)= \cot s$. More generally, we can consider $m_\alpha (s)=e^{- \alpha s} $ with $|\alpha|\leq 1$. Writing $\alpha =\cos \theta$ ($\theta \in [0, +\pi]$)
 we get 
  \begin{itemize}
  \item  for $\alpha =\cos \theta$ with  $\theta \in ]0,\pi [$
  $$ 
   a^{*} (m_\alpha) = \frac{\pi -\theta}{2 \sin \theta}  \,,
   $$
   \item for $\alpha=\pm 1$, 
   $$ 
   a^{*} (m_{-1}) =   \frac 12 \mbox{ and }  a^{*} (m_{ 1} ) =  +\infty  \,.
   $$
  The  global minimizer  restricted to $]0, a^{*}  [$ is given by 
  \begin{equation}
  u_\alpha  (s) = \frac{1}{\sqrt{1-\alpha^2}} \,   \exp( \alpha s) \,
  \sin (\sqrt{1-\alpha^2}\, s),\ -1<\alpha <1
  \end{equation}
  and
  \begin{equation}
  u_{\pm 1} (s) = s \exp \pm s \,.
  \end{equation}
 \end{itemize}
When  $\alpha =\cos \theta$, we get the energy
$$\psi_\alpha (a) = \frac{\sin (\sin \theta a + \theta)}{\sin (\sin \theta a)}.$$
 \end{remark}

 \paragraph{Another upper bound}
 We start from the upper bound 
   $$\int_0^a  (u'^2 -u^2)_+ m(s)^2 ds \leq \int_0^a  u'^2(s)  m(s)^2 \, ds   $$ 
   and minimize the right hand side. \\
   Observing that
   \[
   \begin{split}
   1&=u(a)  = \int_0^a u' (s) \, ds 
   = \int_0^a u' (s)m(s)\, m(s)^{-1}  \, ds\\ &  \leq   \left(\int_0^a (u' (s) m(s)^2 ds\right)^{\frac 12} \,\, \left(\int_0^a \frac{1}{m(s)^2} ds\right)^{\frac 12}\,,
   \end{split}
   \]
   we look for a $u$ for which we have equality.\\
   By the standard Cauchy-Schwarz criterion, this is the case if,  for some constant $C>0$,   $$u' (s) m(s) =\frac{C}{m(s) }\,.$$
   Hence, we choose 
   $$
   u (s) =C   \int_0^s \frac{1}{m(\tau)^2} \, d\tau\,,
   $$
   where the choice of $C$ is determined by imposing $u(a)=1$.
  We obtain
  \begin{proposition}\label{prop3.6}
  For any $a >0$, 
   \begin{equation}
   \inf_{\{u\in H^1(]0,a[), u(0)=0, u(a)=1\}} \int
_0^a (u'^2-u^2)_+m^2\, ds \leq   \left(\int_0^a \frac{1}{m(s)^2}
  ds\right)^{-1}\,.
   \end{equation}
   \end{proposition}
Note here that we have no condition on $a>0\,$ and no condition on $\mu$.

\paragraph{Minimization of $ \exp (2a) I_{\mathrm{inf}}(a)$ } ~\\
In the application to the semi-group upper bound we will meet the natural question of minimizing over $]0, a^{*} ] $
 the  quantity 
 \begin{equation}
a \mapsto  \Theta(a) := \exp (2 a) \,  m^2(a)  \, \psi_0 (a)\,.
 \end{equation}
 The answer is given by the following proposition:
\begin{prop}\label{lem3.17} When $a^{*}(m) <+\infty $, we have
\begin{equation}
 \inf_{a \in ]0,a^{*} ]}  \Theta(a) =\Theta( a^{*}  )=  \exp (2 a^{*}  ) \,  m^2(a^{*} ) \,.
 \end{equation}
 \end{prop}
 
 \begin{proof}
 We will simply show that $\Theta' < 0$ on $]0, a^{*} [$.  Computing $\Theta'$ we get
   \begin{equation}\label{eq:rs1.2}
\Theta'(a) =  \exp (2a)  \,m^2(a)\, \psi_0(a)  (2 + 2 \mu(a) + \psi'_0(a)/\psi_0 (a))
  \end{equation}
Using  \eqref{mh.11},  we obtain  $a<a^{*} $
    \begin{equation}\label{eq:rs1.3}
\Theta'(a) = -  \exp (2a)  \,m^2 (a) ( \psi_0(a) -1)^2 \,.
  \end{equation}
  Note also  that $\Theta'(a^{*} - 0 ) =0$.
 \end{proof}
 
\subsection{Maximizers}\label{ssmax}
As before, let $0<m\in C^1([0,+\infty [)$ and let $0<b<+\infty $. In the following, all
functions are assumed to be real-valued if nothing else is
specified. We recall that $\mathcal G$ was introduced in \eqref{max.1} by
$$
{\cal G}=\{\theta \in H^1(]0,b[);\, |\theta '|\le \theta,\   \theta
(b)=1 \}\,. 
$$

\par If $\theta \in {\cal G}$, we have
$$
|\theta '/\theta |\le 1,\hbox{ i.e.\ }|(\log  \theta )'|\le 1\,,
$$
so $|\log  \theta (s)|\le b-s\,$,
$$
e^{s-b}\le \theta (s) \le e^{b-s},\ 0\le s\le b\, .
$$
In this subsection we consider the problem of maximizing the
functional on $\mathcal G$
\begin{equation}\label{max.2}J(\theta )=J_{]0,b[}(\theta )=\int_0^b (\theta
  ^2-\theta '^2)m^{-2}\,ds\,.
\end{equation}
We recall from (\ref{eq:gp2}) that we have the easy lower bound
$$
J_{\mathrm{sup}} (b):= \sup_{\theta \in \mathcal G} J(\theta) \geq \int_0^b m(s)^{-2}\,ds \,.
$$

We notice that ${\cal G}$ is a bounded subset of
$H^1(]0,b[)$ and that $0\le J\le {\cal O}(1)$ on that subset. As in 
Subsection \ref{exi}  we can show the existence of a
maximizer:
\begin{equation}\label{max.3}
\mbox{There exists } \theta_0 \in {\cal G},\hbox{ such that }J(\theta
_0)=\sup_{\theta \in {\cal G}}J(\theta )\,.
\end{equation}

\par If $0\le \sigma <\tau \le b$, $S,T>0$,  $|\log 
  (T/S)|\le \tau -\sigma $, put
\begin{equation}\label{max.4}
{\cal G }_{\sigma ,\tau }^{S,T}=\{u\in H^1_{S,T}(]\sigma ,\tau [);
|u'|\le u \}.
\end{equation}
We also define
\begin{equation}\label{max.5} 
{\cal G}_\tau ^T=\{u\in H^1_T(]0 ,\tau [);
|u'|\le u \},
\end{equation}
where $H^1_T(]0,\tau [):=\{ u\in H^1(]0,\tau [);\, u(\tau) =T \}$.\\
Finally, we introduce the functional
\begin{equation}\label{max.6}
J_{]\sigma ,\tau [ }(u)=\int_\sigma ^\tau  (u^2-u'^2)m^{-2}ds, \ u\in
H^1(]\sigma ,\tau [)\,.
\end{equation}

\par Let $\theta _0$ be a maximizer for $J$ on ${\cal G}$. If
$0\le \sigma <\tau \le b\,$, we put $S=\theta _0(\sigma )$,
$T=\theta _0(\tau )$. Then ${{\theta _0}_\vert}_{]\sigma ,\tau [}$ is
a maximizer for $J_{]\sigma ,\tau [ }$ on $\mathcal G_{\sigma ,\tau }^{S,T}$. Also
${{\theta _0}_\vert}_{]0,\tau [}$ is a maximizer for $J_{]0,\tau [ }$ on
${\cal G}_\tau ^T$.

\par For $0<  \sigma <\tau \le b$, we assume that $u_0\in
H_{S,T}^1(]\sigma ,\tau [)$ is a maximizer for $J_{]\sigma ,\tau [ }$
on $H^1_{S,T}(]\sigma ,\tau [)$.  Then by the same standard variational
arguments  as for minimizers  (cf. Proposition \ref{mh1}),   we
see that $u_0$ is $1/m$-harmonic  on $]\sigma ,\tau [$:
\begin{equation}\label{max.7}
P_{1/m}u_0:=- (\partial _s\circ m^{-2}\partial _s+m^{-2})u_0=0,\hbox{ on
}]\sigma ,\tau [,
\end{equation}
so $u_0\in H^2(]\sigma ,\tau [)$.

\par When $\sigma =0$, assume that $u_0\in H_T^1(]0 ,\tau [)$ is
a maximizer for $J_{]0,\tau [}$ on $H_T^1(]\sigma ,\tau [)$. Then by
variational calculations, we get
\begin{equation}\label{max.8}
  \begin{cases}
    P_{1/m}u_0=0\hbox{ on }]0,\tau [\,,\\
    \partial _su_0(0)=0,\\
    u_0(\tau )=T.
  \end{cases}
\end{equation}
Also, if $\tau -\sigma >0$ is small enough, we know that
\begin{equation}\label{max.9}m^2P_{1/m}\ge 1/|{\cal O}(1)|\hbox{ on
  }(H^2\cap H_{0,0}^1)(]\sigma ,\tau [),
\end{equation}
and consequently that
\begin{equation}\label{max.10}\begin{split}
&\forall S,T\in \mathbb R,\ \exists ! u=:g_{\sigma ,\tau }^{S,T},\hbox{ such that }\\
&P_{1/m}u=0\hbox{ on }]\sigma ,\tau [,\ u(\sigma )=S,\ u(\tau )=T.
\end{split}
\end{equation}
Similarly to what we have seen in Subsection \ref{ssint},  under
this assumption, $g_{\sigma ,\tau }^{S,T}$ is the unique maximizer for
$J_{]\sigma ,\tau [ }$ on $H^1_{S,T}(]\sigma ,\tau [)$.

\par When $0<\tau \le b$, $m^2P_{1/m}$ is self-adjoint on $L^2(]0,\tau
[,m^{-2}ds)$ with domain ${\cal D}=\{u\in H^1_{T=0}(]0,\tau [)\cap
H^2(]0,\tau[);\, \partial _s u(0)=0 \}$. \\

Moreover, $m^2P_{1/m}\ge 1/|{\cal
  O}(1)|$ when $\tau >0$ is small enough  and for
every $T\in \mathbb R$  we have a unique solution $u=:g^T_{\tau }$ of 
\begin{equation}\label{max.11}
P_{1/m}u=0\hbox{ on }]0,\tau [,\ \partial _su(0)=0,\ u(\tau )=T.
\end{equation}

\par Let $\theta _0$ be a maximizer for $J_{]0,b[}$ on ${\cal
  G}_{0,b}^1$. Let $0<\sigma <\tau \le b$ with $\tau -\sigma \ll 1$
and put $S=\theta _0(\sigma )$, $T=\theta _0(\tau )$. Then $g_{\sigma
  ,\tau }^{S,T}$ is the unique maximizer for $J_{]\sigma ,\tau [ }$ in
$H_{S,T}^1(]\sigma ,\tau [)$. If $g_{\sigma
  ,\tau }^{S,T}$ belongs to the smaller space
${\cal G}_{\sigma ,\tau }^{S,T}$ then it is also the unique maximizer
in that smaller space and we conclude that
\begin{equation}\label{max.12}
{{\theta _0}_\vert}_{]\sigma ,\tau [}=g^{S,T}_{\sigma ,\tau }\,.
\end{equation}

\par Similarly, with $\sigma =0$, if $0<\tau $ is small enough, we see
from (\ref{max.11})
that $g_\tau ^T\in {\cal G}_\tau ^T$, when $T=\theta _0(\tau )>0$. Now
$g^T_{\tau }$ is the unique maximizer for $J_{]0,\tau [ }$ on
$H_T^1(]0,\tau [)$ and a fortiori on ${\cal G}_\tau ^T$, and we conclude that 
\begin{equation}\label{max.13}
{{\theta _0}_\vert}_{]0 ,\tau [}=g^T_\tau .
\end{equation}

\par As above, let $\theta _0$ be a maximizer for $J_{]0,b[}$ on ${\cal
  G}_{0,b}^1$, put
\begin{equation}\label{max.14}
{\widetilde{\phi }} _0=\log \theta _0\,,\ {\widetilde{\psi }} _0= {\widetilde{\phi }} _0'=\theta _0'/\theta _0\,,
\end{equation}
and observe that $|{\widetilde{\psi }} _0|\le 1$. From (\ref{max.13}) we deduce that this inequality is strict near $s=0$.
\begin{lemma}\label{max1} We have $\theta _0'\le 0$, so $-\theta _0\le \theta
  _0'\le 0$, $-1\le {\widetilde{\psi }}_0\le 0$\,.
  \end{lemma}
\begin{proof}
  Assume that $\theta _0'>0$ on a set of positive measure and define
  $\theta _1\in {\cal G}^1_\tau $ by
  $$
\theta _1(b)=1,\ {\widetilde{\psi }} _1:=\theta _1'/\theta _1=-| {\widetilde{\psi }} _0|.
$$
Then $$
\theta _1(t)=\exp \int_b^t {\widetilde{\psi }} _1(s)ds,\ \ \theta _0(t)=\exp \int_b^t {\widetilde{\psi }} _0(s)ds
$$
and 
\begin{equation}\label{max.14.5}
\theta _1(s)\ge \theta _0(s)\,,
\end{equation}
 with strict  inequality near
$s=0$. \\
Now, for $j=0,1$,
$$
\theta _j(t)^2-\theta _j'(t)^2=\theta _j(t)^2(1- {\widetilde{\psi }}_j(t)^2)
$$
where the last factor in the right hand side is independent of
$j$. Hence by \eqref{max.14.5} , we get
$$
\theta _1(t)^2-\theta _1'(t)^2\ge \theta _0(t)^2-\theta _0'(t)^2
$$
and the inequality is strict near $t=0$, so
$J_{]0,b[}(\theta _1)>J_{]0,b[}(\theta _0)$, in contradiction with
the maximality of $\theta _0\,$.  
\end{proof}

We now employ first order ODEs as in Subsubsection \ref{sss3.4.2}. 
Let $f$ be an $1/m$-harmonic function on some interval $]\sigma
,\tau [\subset ]0,b[$ such
that
\begin{equation}\label{max.15}-f< f'\le 0\, .\end{equation}
Put $$\mu =m'/m\,.$$
 Then from $$(\partial _s\circ m^{-2}\circ \partial
_s+m^{-2})f=0\,,$$
we get
\begin{equation}\label{max.16}
(\partial _s^2-2\mu \partial _s+1)f=0\,.
\end{equation}
Writing $${\widetilde{\phi }} =\log  f \mbox{ and }  {\widetilde{\psi }} ={\widetilde{\phi }} '=f'/f\,,$$
we get
\begin{equation}\label{max.17}
 -1<{\widetilde{\psi }} \le 0\mbox{ and } {\widetilde{\phi }} ''+{\widetilde{\phi }} '^2-2\mu {\widetilde{\phi }} '+1=0\,.
\end{equation}
We can rewrite the last equation in the form
  \begin{equation}\label{max.17.5}
      {\widetilde{\psi }} '= 2\mu {\widetilde{\psi }} - {\widetilde{\psi }} ^2 -1,
    \end{equation}
      or equivalently, 
     \begin{equation}\label{max.18}
      {\widetilde{\psi }} '=2\left( \mu -\frac{1}{2}\left( {\widetilde{\psi }} +\frac{1}{{\widetilde{\psi }} } \right)
\right){\widetilde{\psi }}\,.
\end{equation}
Notice that this is the same equation as \eqref{mh.11}, after replacing $\mu$ with $-\mu$.
\par In the region $-1<{\widetilde{\psi }} <0$, we have $(-1/2)({\widetilde{\psi }} +1/{\widetilde{\psi }} )>1$,
hence
$$
2\left( \mu -\frac{1}{2}\left( {\widetilde{\psi }} +\frac{1}{{\widetilde{\psi }} } \right)
\right)>1+\mu \,,
$$
and we conclude that
\begin{equation}\label{eq:3.81new}
{\widetilde{\psi }} '<0,\hbox{ when }-1<{\widetilde{\psi }} <0\hbox{ and }\mu \ge -1.
\end{equation}
When, $\mu <-1$ and $-1<{\widetilde{\psi }} <0$, we have the
  equivalences 
$$
{\widetilde{\psi }} '<0 \Leftrightarrow \mu -\frac{1}{2}\left({\widetilde{\psi }} +\frac{1}{{\widetilde{\psi }} }
\right)>0 \Leftrightarrow g(\mu )<{\widetilde{\psi }} <0,
$$
where $g=g(\mu )$ is the unique solution in $]-1,0[$ of
$$
\mu =\frac{1}{2}\left( g+\frac{1}{g} \right) \mbox{ or equivalently } g^2-2\mu g+1=0\,,
$$
i.e.
\begin{equation}\label{max.19}
g(\mu )=\mu +\sqrt{\mu ^2-1}=\frac{1}{\mu -\sqrt{\mu ^2-1}}\,.
\end{equation}
In other terms, when $\mu <-1$, $-1<{\widetilde{\psi }} <0$, we have
\begin{equation}\label{max.20}
{\widetilde{\psi }} '\ge 0 \hbox{ if and only if }-1<{\widetilde{\psi }} \le g(\mu ).
\end{equation}

\par In all cases, we see directly from (\ref{max.17.5})  that
\begin{equation}\label{max.20a}
{\widetilde{\psi }} '(s)<0,\hbox{ when } |{\widetilde{\psi }} (s)|\le 1/|{\cal O}(1)|,
\end{equation}
so integral curves   of (\ref{max.18})  cannot enter a neighborhood of ${\widetilde{\psi }} =0$ from a
region where ${\widetilde{\psi }} \le -1/C$.

  \begin{remark}\label{max2}
    We have seen that the equations (\ref{mh.11}) and (\ref{max.18})
    differ only by a change of sign of $\mu $. There is a
    corresponding symmetry for the solutions: If $\psi \in C^1(]\sigma
    ,\tau [ ;]0,+\infty [)$, $0\le \sigma <\tau \le +\infty $, then
    \begin{equation}
   \widetilde{\psi }(s):=-1/\psi (s) 
   \end{equation} belongs to the same space and
    \begin{enumerate}
\item $\psi $ solves (\ref{mh.11}) if and only if  $\widetilde{\psi }$ solves
  (\ref{max.18}).
  \item Equivalently, if $u'/u=\psi $, $\theta '/\theta
    =\widetilde{\psi }(=-u/u')$, with $u,\theta >0$, then $u$ is
    $m$-harmonic if and only if  $\theta $ is $1/m$-harmonic.
    \item Pointwise: $\partial _s\psi (s)\ge 0$ $\Longleftrightarrow$ $\partial
      _s\widetilde{\psi }\ge 0$.
\item Pointwise: $1<\psi (s)<+\infty $ $\Longleftrightarrow$
  $-1<\widetilde{\psi }(s)<0$.
  \item We have $\psi (s)\to \infty $ when $s\to \sigma $ if and only if 
    $\widetilde{\psi} (s)\to 0$ when $s\to \sigma $.
    \item
    Let $s_0\in \{\sigma ,\tau  \}$. Then, $\psi (s)\to 1$ when $s\to
    s_0$ if and only if  $\widetilde{\psi }(s)\to -1$ when $s\to s_0$.
    \end{enumerate}
\end{remark}

\paragraph{Structure of maximizers.}
Let us return to the maximizer
$\theta_0 =e^{{\widetilde{\phi }}_0}$ introduced before Lemma \ref{max1}. We know
that $\theta_0$ is $1/m$-harmonic on some interval $]0,\tau[$,
$\tau >0$ and that ${\widetilde{\psi }}_0:= {\widetilde{\phi }}'_0/{\widetilde{\phi }}_0 \in [-1,0[$. From
\eqref{max.20}, we see that ${\widetilde{\psi }}'_0 <0$ near $0$ and
$$
-1 \leq {\widetilde{\psi }}_0 (s) \leq -1 /|\mathcal O (1)|\,,$$ on
$]\epsilon,b[$, for every $\epsilon >0$. Thus
whenever $\theta_0$ is $1/m$-harmonic  on a subinterval $\subset]\epsilon,b[$, we have the
differential equation \eqref{max.18} (with ${\widetilde{\psi }}$ replaced by
${\widetilde{\psi }}_0$) with a nice uniform control (no blow up). Also
\begin{equation}\label{max.21}
  ({\widetilde{\phi_0})_{\vert _{]\epsilon ,b[ } } } \in H^1(]\epsilon ,b[),
\end{equation}
for every $\epsilon >0$.
\par As in Subsection \ref{stm} we have 
\begin{prop}\label{max3}
Let $0<\sigma <\tau \le b$ and let us assume that \break  $ \lambda :=m_{]\sigma ,\tau [}({\widetilde{\psi }}
_0)> - 1\,$. Then there exists $ s_0\in ]\sigma ,\tau [$ and $\alpha ,\beta $
with
\begin{equation}\label{max.22}
0\le \alpha <s_0<\beta \le b\,,
\end{equation}
such that
\begin{equation}\label{max.23}
\theta_0 \hbox{ is } 1/m\hbox{-harmonic and} -1<{\widetilde{\psi }} _0<0 \hbox{ on }]\alpha ,\beta [\,.
\end{equation}
\begin{itemize}
\item 
If $\beta <b$, we have ${\widetilde{\psi }} _0(s)\to -1$, $s\nearrow \beta $\,.
\item If $\alpha >0$, we have ${\widetilde{\psi }} _0(s)\to -1$, $s\searrow \alpha\,.  $
\end{itemize}
We recall that ${\widetilde{\psi }} _0(s)\to 0 $, when $s\searrow 0\,$.\\
Moreover $s_0$ can be chosen so that ${\widetilde{\psi }} _0(s_0)$ is arbitrarily close to
$\lambda \, $.
\end{prop}

Let $J \subset ]0,b[$ be the countable disjoint union of all open maximal intervals $I \subset ]0,b[ $, such that $\theta_0$ is $1/m$-harmonic and $-1 <  {\widetilde{\psi }}_0 <1$ on $I$.
\begin{prop}\label{max5}
${\widetilde{\psi }} _0$ is  uniformly  Lipschitz continous on $]0,b]$, $ >-1$ on $J$, and $=-1$ on
$]0,b[\setminus J$.
\end{prop}

Using Remark \ref{max2}, we can carry over the results about
minimizers $u_0$ on maximal subintervals where $u_0'/u_0>1$, to
maximizers $\theta _0$ on
maximal subintervals where $\theta _0$ is $1/m$-harmonic with
$-1<\theta _0'/\theta _0<0$. Thus for instance we have

\begin{prop}\label{max6}
Assume that $\mu \ge -1$ on $[0,b]$ and let $\theta _0$ be a maximizer
for $J_{]0,b[}$ on ${\cal G}={\cal G}_b^1$. Then there exists
$\widetilde{b}\in ]0,b]$ such that
\begin{itemize}
\item[]$\theta _0\in C^2([0,\widetilde{b}])$, $\theta _0'(0)=0$,
  \item[]$\theta _0$ is $m$-harmonic and $-1<\theta _0'/\theta _0<0$
    on $]0,\widetilde{b}[$,
    \item[] $\theta _0(s)=e^{b-s}$ on $]\sigma ,b[$ (if
      this interval is $\ne \emptyset $).
\end{itemize}
\end{prop}

\par We end this subsection with a discussion of global maximizers.
Let
$$
{\cal G}(]0,+\infty [):=\{u\in H^1_{\mathrm{loc}}([0,+\infty [);\,
0\le u'\le u ,\ u>0\hbox{ on }]0,+\infty [\}.
$$
We say that $\theta _0\in {\cal G}(]0,+\infty [)$ is a maximizer (or a
global maximizer when emphasizing that we work on the whole half axis)
if ${{\theta _0}_\vert}_{]0,b[}$ is a maximizer in ${\cal
  G}_b^{\theta _0(b)}$ for every $a>0$.
\begin{prop}\label{max7}
A global maximizer $\theta _0\in {\cal G}(]0,+\infty [)$ exists.
\end{prop}
Indeed, the proof of Proposition \ref{stm7} applies with minor
changes.\\

\par The discussion of the structure of maximizers in ${\cal G}_b^1$
carries over directly to that of global maximizers. In particular, if
$\theta _0$ is a global maximizer, then $\theta _0$ is $1/m$-harmonic
with $-1<\theta _0'/\theta _0<0$ on a maximal interval interval of the
form $]0,b^{*} [$ for some $b^{*} \in ]0,+\infty ]$. $b^{*} $ is
uniquely determined and (the $1/m$-harmonic function) ${{\theta
    _0}_\vert}_{]0,b^{*} [}$ is unique up to a constant positive factor.

\par By Remark \ref{max2}, we have
\begin{equation} \label{eq:bstar=astar}b^{*} =a^{*} .
\end{equation}

$b^{*} $ is also characterized as the largest number in $]0,+\infty ]$
such that the smallest eigenvalue of $K^{N\widetilde{R}}_{1/m,b}$ is
$>0$ for $b<b^{*} $. Here $K^{N\widetilde{R}}_{1/m,b}$ is defined as
in the introduction, with $m$ replaced by $1/m$ and with the
domain
$$
{\cal D}(K^{N\widetilde{R}}_{1/m,b})=\{
u\in H^2(]0,b[);\, u'(0)=0,\ u'(b)=-u(b)
\} .
$$
 
As for the minimization problem, we have 
  \begin{prop}\label{prop3.26}
For $\mathbb R\ni b \in ]0,a^{*}]$, we have 
 \begin{equation}
\sup_{{\cal G}} \int
_0^b (\theta^2-\theta'^2)m^{-2} \,ds =  -\frac{\widetilde{\psi }_0(b)}{m(b)^2}=  \frac {1}{m(b)^2} \frac{1}{\psi_0(b)} \,.
\end{equation}
In particular, when $b=a^{*}<+\infty  $: 
  \begin{equation}
  \sup_{{\cal G}} \int
_0^b (\theta^2-\theta'^2)m^{-2} \,ds  =  \frac {1}{m(a^{*} )^2}\,.
\end{equation}  
  \end{prop}
\begin{proof}
Similarly to the proof of Proposition \ref{prop3.18}, we can this time start from the global maximizer $\theta_0$ and compute for $b\leq b^{*} $ the integral  $ \int
_0^b (\theta^2-\theta'^2)m^{-2} \,ds$ with $\theta (s) = \theta_0 (s)/\theta_0 (b)$. We obtain
\begin{equation}
 \int
_0^b (\theta^2-\theta'^2)m^{-2} \,ds = - \theta'_0(b)\, m(b)^{-2}\,.
\end{equation}
We then use \eqref{eq:bstar=astar} and Remark \ref{max2}.
\end{proof}
\begin{remark}
In the case when $m=1$. We have  $b^{*} =\frac \pi 4$. 
 $$
 \theta_0 (s) = \sqrt{2} \cos s   \,.
 $$
 The corresponding energy is under the condition $0 < b  \leq \frac \pi 4$,
 $$
 \int_0^{b }  (\theta(s)^2-\theta'(s)^2) \,ds = \tan  b  \,.
 $$
 \end{remark}

\section{Optimization in Th. \ref{Th1.2}: case
  $\epsilon_1= - \epsilon_2 = +$.} \label{s4} \setcounter{equation}{0} 
\subsection{Reduction to  $\omega=0$ and $ r(0)=1$}\label{ss4.1new}
   
  Let $A$, $r=r(\omega )$, $\omega $ be as in Theorem \ref{Th1.2} and 
  (\ref{defr}). Let $\hat {\omega }\in \mathbb R$,
  $\hat {r}=\hat {r}(\hat {\omega })>0$. Then
  $(\hat {A}, \hat {r}(\hat {\omega
  }),\hat {\omega })$ has the same properties, if we define
  $\hat {A}$ by
  $$
\frac{1}{r}(A-\omega
)=\frac{1}{\hat {r}}(\hat {A}-\hat {\omega }).
$$
Notice here that (\ref{defr}) can be written
$$
1=\sup_{\Re w>0}\| r(\omega )(A-\omega -w)^{- 1}\|
$$
and that
$$
r(\omega )(A-\omega -w)^{-1}=\hat {r}(\hat {\omega
})(\hat {A }-\hat {\omega }-\hat {w})^{-1},
$$
if $\hat {w}/\hat {r}=w/r$.

Let $S(t)=\exp ( tA)$, $\hat {S}(\hat {t})=\exp
(\hat {t}\hat {A})$, $t,\hat {t}\ge 0$. If $\|S(t)\|\le
m(t)$ for some $t\ge 0$, then $\|\hat {S}(\hat {t})\|\le \hat {m}(\hat {t})$
if
$$
\frac{\hat {m}(\hat {t})}{e^{\hat {\omega} \hat {t}}}=\frac{m(t)}{e^{\omega t}},\ \ \hat {r}\hat {t}=rt.
$$
This follows from,
$$
e^{-\omega t}S(t)=\exp t(A-\omega )=\exp
\hat {t}(\hat {A}-\hat {\omega })=e^{-\hat {\omega
  }\hat {t}}\hat {S}(\hat {t}).
$$

Theorem
  \ref{Th1.2} tells us that if $\|S(t)\|\le m(t)$, $t\ge 0$, then
  $\|S(t)\|\le m_\mathrm{new}(t)$, for $t\ge 0$, where
  \begin{equation}\label{LLa.5}
 \frac{m_\mathrm{new}(t)}{e^{\omega t}}= \frac{\|(r(\omega
  )^2\Phi^2 -\Phi '^2)^{\frac{1}{2}}_-\,m/e^{\omega \cdot }\|_{L^2([0,t[) } \|(r(\omega)^2 \Psi
  ^2- \Psi '^2)^{\frac{1}{2}}_-\, m/e^{\omega \cdot }\|_{L^2([0,t[) }}
{\int_0^t (r(\omega
    )^2\Phi^2 -\Phi '^2)^{\frac{1}{2}}_{\epsilon_1}(r(\omega )^2\iota_t \Psi^2-\iota_t \Psi
    '^2)^{\frac{1}{2}}_{\epsilon_2} ds}\,.
\end{equation}
With $\hat {\Phi }(\hat {t})=\Phi (t)$,
$\hat {\Psi }(\hat {t})=\Psi (t)$, we have
$\Phi '(t)/r(\omega )=\hat {\Phi
}'(\hat {t})/\hat {r}(\hat {\omega })$ and similarly
for $\Psi ' $, $\hat {\Psi }'$. If
$\hat {m}_\mathrm{new}(\hat {t})$ is defined by
$\hat {m}_\mathrm{new}(\hat {t})/e^{\hat {\omega
  }\hat {t}}= m_{\mathrm{new}}(t)/e^{\omega t}$, then
(\ref{LLa.5}) implies the analogous relation for
$\hat {m}_{\mathrm{new}}$:
\begin{equation}\label{LLb.5}
 \frac{ \hat {m}_\mathrm{new}(\hat {t})
 }{e^{\hat {\omega }\hat {t}}}=
  \frac{\|(\hat {r}(\hat {\omega })^2\hat {\Phi }^2 -\hat {\Phi }
    '^2)^{\frac{1}{2}}_-\,\hat {m}/e^{\hat {\omega }\cdot }\|_{L^2([0,\hat {t}[) }
    \|(\hat {r} (\hat {\omega })^2\hat {\Psi }
    ^2- \hat {\Psi  }'^2)^{\frac{1}{2}}_-\, \hat {m}/e^{\hat {\omega }\cdot }\|_{L^2([0,\hat {t}[) }}
  {\int_0^{\hat {t}} (\hat {r}(\hat {\omega })^2\hat {\Phi }^2 -\hat {\Phi }
    '^2)^{\frac{1}{2}}_{\epsilon_1}
   ( \hat {r} (\hat {\omega })^2\iota_{\hat {t}} \hat {\Psi }^2-\iota_{\hat {t}} \hat {\Psi }
    '^2)^{\frac{1}{2}}_{\epsilon_2} d\hat {s}}\,.
\end{equation}
We also saw above that $\|\hat {S}(\hat {t})\|\le
\hat {m}_\mathrm{new}(\hat {t})$. Thus if we have proved
Theorem~\ref{Th1.2} for $(A,\omega ,r,m)$ we get it also for
$(\hat {A},\hat {\omega },\hat {r},\hat {m})$, and
vice versa. In particular we could reduce the proof of the theorem to
the special case when $\omega =0$, $r(\omega )=1$.

\par We review the above scaling in a slightly special case, keeping
an eye on the scaling of some optimizers from Section \ref{s3}.
 Let $\hat {A}$, $\hat {r}=\hat {r}(\hat {\omega }
 )$, $\hat {\omega } $ be as in Theorem \ref{Th1.2} and 
  (\ref{defr}), where we have added hats for notational
  convenience. Let
  $$A=\frac{1}{\hat {r}(\hat {\omega
    })}(\hat {A}-\hat {\omega }).$$
  As above, we check that $A$ satifies the general assumptions with
  $\omega =0$, $r=r(\omega )=1$. With $t=\hat {r}\hat {t}\ge
  0$, we have
  $$
\|e^{tA}\|\le m(t)\Leftrightarrow
\|e^{\hat {t}\hat {A}}\|\le \hat {m}(\hat {t)},
$$
if $m(t)>0$, $\hat {m}(\hat {t})>0$ are related by
$$
m(t)=e^{-\hat {t}\hat {\omega
  }}\hat {m}(\hat {t}),\hbox{ or equivalently
}\hat {m}(\hat {t})=e^{t\hat {\omega }/\hat {r}}m(t).$$
Theorem \ref{Th1.2} applies to
$\hat {S}(\hat {t})=e^{\hat {t}\hat {A}}$. It is a
little more scale invariant to rewrite (\ref{LL.5}) as
\begin{equation}\label{s4.1}
e^{-\hat {\omega
  }\hat {t}}\|e^{\hat {t}\hat {A}}\|\le
\frac{\|(\hat {\Phi }^2-(\hat {\Phi
  }'/\hat {r})^2)_-^{1/2}e^{-\hat {\omega }\cdot
  }\hat {m}\|_{[0,\hat {t}]}
\|(\hat {\Psi }^2-(\hat {\Psi
  }'/\hat {r})^2)_-^{1/2}e^{-\hat {\omega }\cdot }\hat {m}\|_{[0,\hat {t}]}
}
{\int_0^{\hat {t}}(\hat {\Phi }^2-(\hat {\Phi
  }'/\hat {r})^2)_{\epsilon _1}^{1/2}
((\iota _{\hat {t}}\hat {\Psi })^2-(\iota _{\hat {t}}\hat {\Psi
  }')/\hat {r})^2)_{\epsilon _2}^{1/2} d\hat {s}
},
\end{equation}
where the
subscript $[0,\hat {t}]$ indicates the interval over which we
take the $L^2$-norm.

\par Putting $s=\hat {r}\hat {s}$, $\Phi (s)=\hat {\Phi
}(\hat {s})$, $\Psi (s)=\hat {\Psi }(\hat {s})$, we get
$\hat {\Phi }'/\hat {r}=\Phi '$, $\hat {\Psi
}'/\hat {r}=\Psi '$,
\begin{equation}\label{s4.2}
e^{-\hat {\omega
  }\hat {t}}\|e^{\hat {t}\hat {A}}\|\le
\frac{\|(\Phi ^2-\Phi
  '^2)_-^{1/2}m\|_{[0,t]}
\|(\Psi ^2-\Psi '^2)_-m\|_{[0,t]}
}
{\int_0^t(\Phi ^2-\Phi
'^2)_{\epsilon _1}^{1/2}
(\iota _{t}\Psi ^2-\iota _{t}\Psi '^2)_{\epsilon _2}^{1/2} ds
}.
\end{equation}
In (\ref{red.1}) we studied the minimization of a factor in the
enumerator,
\begin{equation}\label{s4.3}
\inf_{u\in H^1_{0,1}(]0,a[)}I(u),\hbox{ where }I(u)=I_{]0,a[}(u)=\int
_0^a (u'^2-u^2)_+m^2ds.
\end{equation}
The corresponding problem appearing in (\ref{s4.1}) is
\begin{equation}\label{s4.4}
\inf_{\hat {u}\in H^1_{0,1}(]0,\hat {a}[)}\hat {I}(\hat {u}),\hbox{ where }\hat {I}(\hat {u})=\hat {I}_{]0,\hat {a}[}(\hat {u})=\int
_0^{\hat {a}}
((\hat {u}'/\hat {r})^2-\hat {u}^2)_+(e^{-\hat {\omega
  }\cdot }\hat {m})^2d\hat {s}.
\end{equation}
$u$ is a minimizer for (\ref{s4.3}) iff $\hat {u}$ is a minimizer
for (\ref{s4.4}) when $u$, $\hat {u}$ are related by
\begin{equation}\label{s4.5}
\hat {u}(\hat {s})=u(s).
\end{equation}

\par We have seen that a minimizer $u$ for (\ref{s4.3}) belongs to the
space
$$
{\cal H}_{0,1}(]0,a[)=\{u\in H^1(]0,a[);\, 0\le u\le u' \}.
$$
The corresponding space for (\ref{s4.4}) is then
$$
\hat {{\cal H}}_{0,1}(]0,\hat {a}[)=\{\hat {u}\in H^1(]0,\hat {a}[);\, 0\le \hat {u}\le \hat {u}'/\hat {r} \}.
$$
We have seen in Subsection \ref{stm} that $I$ has an associated global
minimizer $u$ which is $m$-harmonic with $u'>u$ on $]0,a^*[$ and when
$a^*<+\infty $ we have $u'(a^*)=u(a^*)$. Moreover $a^*$ is uniquely
determined, and up to multiplication with a positive constant, the same
holds for ${{u}_\vert}_{]0,a^*[}$. Similarly we have a global
minimizer $\hat {u}$ associated to $\hat {I}$, related to a
global minimizer $u$ via (\ref{s4.5}). The corresponding variational
equation on any open interval where $0\le
\hat {u}<\hat {u}'/\hat {r}$, is
\begin{equation}\label{s4.6}
\left(\frac{1}{\hat {r}}\partial _{\hat {s}}\circ
\left(e^{-\hat {\omega
    }\hat {s}}\hat {m}(\hat {s}) \right)^2
\frac{1}{\hat {r}}\partial _{\hat {s}}+
\left(e^{-\hat {\omega
    }\hat {s}}\hat {m}(\hat {s}) \right)^2\right)
\hat {u}=0 .
\end{equation}
This holds on $]0,\hat {a}^*[$, where
$a^*=\hat {r}\hat {a}^*$ and when $a^*<\infty $, we have
$$\hat {u}'(\hat {a}^*)/\hat {r}=\hat {u}(\hat {a}^*).$$

\par In Subsection \ref{sss3.4.2} we studied a Riccati equation for an
$m$-harmonic function $u$ in terms of the logarithmic derivative $\psi =u'/u$. In the case
of (\ref{s4.6}) with general $\hat {r}$, $\hat {\omega }$,
the natural logarithmic derivative is $\hat {\psi }=(\hat {u}'/\hat {r})/\hat {u}$. \\
 In conclusion  Theorem \ref{prop4.9bis} is a direct consequence of Theorem \ref{prop4.9}.

\subsection{ Other preliminaries}
We now  assume $\omega=0$ and $ r(0)=1$.
In this case, \eqref{LL.5} takes the form
\begin{equation}\label{LL.5w}
 || S (t)|| _{\mathcal L(\cal H)}\le  \frac{\| (\Phi^2-\Phi'^2)^{\frac{1}{2}}_-\,m\|_{L^2(]0,t[)} \, \| \Psi^2 -( \Psi ')^2)^{\frac{1}{2}}_-\, m\|_{L^2(]0,t[)}  }
{\int_0^t (\Phi^2-(\Phi ')^2)^{\frac{1}{2}}_{+}\,((\iota_t\Psi)^2-((\iota_t \Psi)
    ')^2)^{\frac{1}{2}}_{-} ds}\,.
\end{equation}
Replacing $(\Phi, \Psi)$ by $(\lambda \Phi,\mu \Psi)$ give for any $(\lambda,\mu)\in( \mathbb R\setminus \{0\})^2$ does not change the right hand side. Hence we may  choose a suitable normalization without loss of generality. We also choose $\Phi$ and $\Psi$ to be piecewise $C^1([0,t]))$ (see Footnote \ref{fnpiecewise} for the definition).\\
Given some $t>a + b $, we now give the conditions satisfied by $\Phi$:
\begin{property}[$P_{a,b}$]~
  \begin{enumerate}
  \item $\Phi = e^a u $ on $]0,a]$ and $u\in \mathcal H:= \mathcal H_{0,a}^{0,1}$ (cf \eqref{red.2})\footnote{Here is our choice of normalization}\,.
  \item On $[a,t - b  ]$, we take  $\Phi(s) =e^{s}$, so
  $\Phi'^2(s) -\Phi(s)^2=0\,.$
  \item  On $[t-  b  , t]$ we take $\Phi(s)=e^{t-b }\theta (t-s)$ with $\theta \in \mathcal G=\mathcal{G}_b^1$ \,.
  \end{enumerate}
  \end{property}
  Hence, we have 
  $$
  {\rm Supp} ( \Phi^2 -\Phi'^2)_+\subset [ t-b , t]\,.
  $$
 Similarly we assume that $\Psi$ satisfies property $(P_{b,a})$ but
 with $\theta=1$,  hence
 \begin{enumerate}
 \item 
 $
 \Psi(s) = e^{b} v(s) \mbox{ on } ]0,b[  \mbox{ with }  v  \in \mathcal H_b \,,
 $ 
 where 
 $
\mathcal H_b:=  \mathcal H_{0,b}^{0,1} \,.
$
\item    On $[b,t-a  ]$, we take  $\Psi(s) =e^{s}$.
\item On $[t-a ,t ]$,  $\Psi (s) = e^{t-a}\,.$
\end{enumerate}

 Recalling the definition of $\iota_t$, we get for $\iota_t\Psi$:
  \begin{enumerate}
  \item On $[0,a]$, $\iota_{t} \Psi= e^{(t- a )}$, satisfying 
  $$(\iota_t \Psi) '^2 -(\iota_t \Psi)^2 =  - e^{2(t-a)}\,.$$
  \item On $[a,t -  b  ]$, we have  $\iota_t \Psi (s) =e^{ t-s}$, hence 
  $$(\iota_t \Psi)'(s)^2 -\iota_t \Psi(s)^2=0\,.$$
  \item  On $]t-   b  , t[$, we have 
  $$(\iota_t \Psi) '(s)^2 -(\iota_t \Psi)^2(s) \geq  0 \,.
  $$
  \end{enumerate}
 Recall our choice of $\epsilon_1=+$ and $\epsilon_2=-$.  Assuming that $t>a + b $, we have  under these assumptions on $\Phi$ and $\Psi$
$$  
  \{ s;\Phi(s)^2-(\Phi '(s))^2 >0,\iota_t \Psi(s)^2 -(\iota_t \Psi '(s))^2 < 0  \} \subset [ t-b,b]\,.
  $$
 We now compute or estimate the various quantities appearing in \eqref{LL.5w}.\\
  
We have
\begin{equation}\label{eq:Iphi}
  \| (\Phi^2-\Phi'^2)^{\frac{1}{2}}_-\,m\| = e^{a}\left(  \int_0^a
    (u'(s)^2 - u^2(s)) m(s)^2 ds
  \right)^{1/2}\,,
\end{equation}
\begin{equation}\label{eq:Ipsi}
\|(\Psi^2-(\Psi ')^2)^{\frac{1}{2}}_-m\| =  e^{b}\left( \int_0^{b }
  (v'(s)^2 - v(s)^2) m(s)^2 ds\right)^{1/2}
  \,,
\end{equation}
and 
\begin{equation}\label{eq:Ichi}
\begin{split}
&\hskip -1cm\int_0^t(\Phi^2-\Phi'^2)^{\frac{1}{2}}_+ ((\iota_t \Psi)^2- (\iota_t \Psi')^2 )^{\frac{1}{2}}_-ds \\ 
&=  \int_{t-b} ^t(\Phi^2-\Phi'^2)^{\frac{1}{2}}_+ ((\iota_t \Psi)^2-
  (\iota_t \Psi')^2)^{\frac{1}{2}}_-ds \\ 
 &= e^{t-b } \int_{t-b }^t   (\theta(t-s) ^2- \theta'(t-s)^2)_+^{\frac 12}  \,   ((\iota_t \Psi)^2- (\iota_t \Psi')^2)^{\frac{1}{2}}_- ds
\\ 
& = 
  e^{t} \int_{0}^{b }   (\theta (s) ^2- \theta'(s)^2)^\frac 12  (v'(s)^2 - v (s)^2)^{\frac{1}{2}} ds \,.
  \end{split}
\end{equation}
So we get from \eqref{LL.5w}
\begin{equation}\label{LL.5z}
 || e^t  S (t)|| _{\mathcal L(\cal H)}\le  e^{a+ b}   \left (\int_0^a (u'(s)^2 - u^2(s)) m(s)^2 ds\right)^\frac 12 \, K (b,\theta,v) \,,
\end{equation}
where 
\begin{equation}\label{Kbthetav}
 K (b,\theta,v) := \frac{ \left(\int_0^{b } (v'(s)^2 - v(s)^2) m(s)^2
     ds\right)^{1/2} }{ \int_{0}^{b }   (\theta (s) ^2- \theta'(s)^2)^\frac 12  (v'(s)^2 - v (s)^2)^{\frac{1}{2}} ds }\,.
 \end{equation}
 We start by considering for a given $\theta \in \mathcal G$
 $$
K_{\mathrm{inf}} (b,\theta) := \inf_{ v \in \mathcal H_b} K(b,\theta ,v)\,,
 $$
and  get  the following:
  \begin{lemma}\label{prop3.12} If $\theta\in \mathcal G$ and $\theta
    -\theta '$ is not identically $0$ on $]0,b$[, we have
\begin{equation}\label{idkinfbtheta}
K_{\mathrm{inf}} (b,\theta)  = \frac{1}{\sqrt{\int_0^{b }   (\theta (s) ^2- \theta'(s)^2) \frac {1}{ m^2} ds}} \,.
\end{equation}
\end{lemma}
\begin{proof}
Inspired by the proof  in Subsection \ref{ss3.6}, we consider
with $$h(s)=  (\theta (s) ^2- \theta'(s)^2) ^\frac 12  \geq 0 $$ the
denominator in (\ref{Kbthetav}), 
$$
\int_0^b h(s)  (v'(s) ^2 - v(s)^2)^{\frac{1}{2}} ds =\int_0^b ( h(s)/m(s) ) \,(   m(v'(s) ^2 - v(s)^2)^{\frac{1}{2}}) ds  \,.
$$
By the Cauchy-Schwarz inequality, we have
\begin{multline*}
\int_0^b h(s) (v'(s) ^2 - v(s)^2)^{\frac{1}{2}} ds \leq \\ \left(
  \int_0^b ( h(s)/m(s) )^2ds\right)^\frac{1}{2} \, \left(\int_0^b ( m(s)^2(v'(s) ^2 -
  v(s)^2) ds \right) ^\frac 12 \,,
\end{multline*}
which implies that $K_{\mathrm{inf}}(b,\theta )$ is
  bounded from below by the right hand side of (\ref{idkinfbtheta}).

 We have equality for some $ v$ in $\mathcal H_b$ if and only if 
$$
 m(s) (v'(s)^2 - v(s)^2)^{\frac{1}{2}} = c \,  h(s)/m(s)  \,.
 $$
 for some constant $c>0$. In order to get such a $v$,
we first consider $w \in H^1$ defined by
$$
w' = \sqrt{ w^2  + h^2m^{-4}}\,,\, w(0)=0\,,
$$
noticing that the right hand side of the differential
  equation is Lipschitz continuous in $w$, so that the
  Cauchy-Lipschitz theorem applies.
According to our assumption on $\theta$, we verify that $w(b) >0$ 
 and we choose
 $$
v = \frac{1}{w(b)} w\,,\,c = \frac{1}{w(b) }\,.
$$
For this pair $(c,v)$ we get
\begin{equation}\begin{split}
&\left(\int_0^{b } (v'(s)^2 - v(s)^2) m(s)^2 ds\right)^\frac 12 \Big/
  \int_{0}^{b }  \left(\theta (s) ^2- \theta '(s)^2 \right) ^\frac 12
  \left(v'(s)^2 - v(s)^2\right)^{\frac{1}{2}} ds\\
  &=1 \big/\left( \int_0^b h^2(s)  m^{-2}(s)  \, ds \right)^\frac 12   \,.
\end{split}
\end{equation}
Returning to the definition of $h$ shows that
  $K_\mathrm{inf}(b,\theta )$ is bounded from above by the right hand
  side of (\ref{idkinfbtheta}) and we get the announced result.
\end{proof}
 
 To conclude the proof of  Proposition \ref{propminmax}, we just
   combine Lemma \ref{prop3.12}  and   \eqref{LL.5z}.

    \appendix
\section{ Appendix: Optimization  with 
  $\epsilon_1=\epsilon_2=+$}\label{++}\setcounter{equation}{0}

In this section we let $\epsilon _1=\epsilon _2=+$ in Theorem
\ref{Th1.2} and assume that  
$\mathrm{supp\,}(r(\omega )^2-\phi '^2)_-\subset [0,a]$,
$\mathrm{supp\,}(r(\omega )^2-\psi '^2)_-\subset[0,b]$ for some
$a,b>0$, where $\Phi =e^\phi $, $\Psi =e^\psi $.  The results we
  get in this case seem less decisive, but perhaps still of some interest.
Assuming, to start with, that $\phi $ and $\psi $ are given on
$[0,a]$ and $[0,b]$ respectively, we shall discuss how to choose $\phi
$ on $]a,+\infty [$ and $\psi $ on $]b,+\infty [$, for every given $t>a+b$,  in order to optimize
the estimate on $\| S(t)\|$. A later problem 
will be to choose $a,b$ with $a+b<t$ and the restrictions of
$\phi $, $\psi $ to $[0,a]$ and $[0,b]$ respectively. 

\par From (\ref{LL.5}) we get with $r=r(\omega )$,
\begin{equation}\label{++.1}
\|S(t)\|\le e^{\omega t} \frac{\| (r^2-\phi '^2)_-^{1/2}m\|_{\phi
    -\omega \cdot } \|
  (r^2-\psi '^2)_-^{1/2}m\|_{\psi -\omega \cdot }}{I(\phi ,\psi )},
\end{equation}
where
\begin{equation}\label{++.2}
I(\phi ,\psi )=I_{a,b,t}(\phi ,\psi )=\int_a^{t-b} e^{\phi +\iota
  \psi }(r^2-\phi '^2)_+^{1/2}(r^2-\iota \psi '^2)_+^{1/2} ds.
\end{equation}
Here we put $\iota_t \psi (s)=\psi
(t-s)$,  and write simply $\iota $ when the choice of $t$ is clear.
 We try to choose $\phi (s) $ for $s\ge a$ and $\psi (s)$
for $s\ge b$ so that $I(\phi ,\psi )$ is as large as possible. 
 Write
$$
\phi (s)=\phi (a)+\widetilde{\phi }(s-a),\ \psi (s)=\psi
(b)+\widetilde{\psi }(s-b).
$$
For $s\in [a,t-b]$, set $s=a+\widetilde{s}$, $0\le \widetilde{s}\le
  t-a-b$. Then with $\widetilde{t}=t-a-b$,
\begin{multline*}
\phi (s)+\iota \psi (s)=\phi (s)+\psi (t-s)=\phi
(a+\widetilde{s})+\psi (t-a-\widetilde{s})\\
=\phi (a)+\psi (b)+(\phi (a+\widetilde{s})-\phi (a))+(\psi
(b+(t-a-b)-\widetilde{s})-\psi (b)\\
=\phi (a)+\psi (b)+\widetilde{\phi }(\widetilde{s})+\widetilde{\psi }
(t-a-b-\widetilde{s})\\=
\phi (a)+\psi (b)+\widetilde{\phi }(\widetilde{s})+\iota
_{\widetilde{t}}\widetilde{\psi }(\widetilde{s})
\end{multline*}
and we get with $\widetilde{\iota } =\iota _{\widetilde{t}}$,
\begin{equation}\label{++.3}\begin{split}
I(\phi ,\psi )=e^{\phi (a)+\psi (b)}&\int_0^{\widetilde{t}}
e^{\widetilde{\phi }(\widetilde{s })+\widetilde{\iota }
  \widetilde{\psi }(\widetilde{s})}
(r^2-\widetilde{\phi }'^2)_+^{1/2}
(r^2-\widetilde{\iota } \widetilde{\psi
} '^2)_+^{1/2}d\widetilde{s}\\
=& e^{\phi (a)+\psi (b)}I_{0,0,\widetilde{t}}(\widetilde{\phi
},\widetilde{\psi }).
\end{split}
\end{equation}
We wish to choose $\widetilde{\phi }$, $\widetilde{\psi }$ with
$\widetilde{\phi }(0)=\widetilde{\psi }(0)=0$ such that
$\widetilde{I}(\widetilde{\phi } ,\widetilde{\psi
})=I_{0,0,\widetilde{t}}(\widetilde{\phi },\widetilde{\psi })$ is as
large as possible. 

Drop the tildes for a while. The problem is then to choose $\phi $, $\psi $
with $\phi (0)=\psi (0)=0$ such that 
\begin{equation}\label{++.4}
I(\phi ,\psi )=\int_0^t e^{\phi +\iota \psi }(r^2-\phi '^2)_+^{1/2}
(r^2-\iota \psi '^2)_+^{1/2} ds
\end{equation}
is as large as possible. 

At this moment we do not know how to solve this
general problem, so we restrict the class of functions (satisfying
$\phi (0)=\psi (0)=0$) by requiring that
\begin{equation}\label{++ny.1}
\phi +\iota \psi =\mathrm{Const.}\hbox{ on }[0,t].
\end{equation}
In other words, we require that $(\iota \psi )'=-\phi '$. The constant
in (\ref{++ny.1}) is then equal to $\phi (t)=\int_0^t \phi
'(s)ds$. With $\|\cdot \|$ denoting the $L^2([0,t])$-norm, we get
\begin{equation}\label{++ny.2}
I(\phi ,\psi )=\exp \left(\int_0^t \phi '(s)ds\right)\int_0^t
(r^2-\phi '(s)^2)ds\le \exp (\sqrt{t}\|\phi '\|)(tr^2-\|\phi '\|^2),
\end{equation}
requiring also the $|\phi '|\le r$.
Here we have equality precisely when $\phi '$ is equal to some
constant $\alpha \in [0,r]$, so for any given value $\beta \in [0,r\sqrt{t}]$ of
$\|\phi '\|$, we should choose
\begin{equation}\label{++.5}
   \phi '=\alpha, \ \
\phi (s)=\psi (s)=\alpha s,\hbox{ with }
\sqrt{t}\alpha =\beta .
\end{equation}
The corresponding maximal value of
$I(\phi ,\psi )$ is given by
\begin{equation}\label{++.6}
J(\alpha, t)=te^{\alpha t}(r^2-\alpha ^2)
\end{equation}
We look for the maximum of this function of $\alpha $:
$$
\partial _\alpha J(\alpha ,t)=-t^2e^{\alpha t}\left(\alpha ^2+\frac{2}{t}\alpha -r^2 \right)
$$
The two critical points are given by a local maximum at 
\begin{equation}\label{++.7}
\alpha _+=\alpha _+(t)=\frac{1}{t}(\sqrt{1+(rt)^2}-1)\in ]0,r[
\end{equation}
and a local minimum at
$$
\alpha _-=\alpha _-(t)=-\frac{1}{t}(\sqrt{1+(rt)^2}+1)<0.
$$
We see that $\alpha _+$ is a global maximum. The corresponding maximal
value is given by
\begin{equation}\label{++.7.5}
J_{\mathrm{max}}(t)=J(\alpha _+,t)=e^{\sqrt{1+(rt)^2}-1}\frac{2}{t}(\sqrt{1+(rt)^2}-1)
\end{equation}

{}

\par Let us compute the asymptotic behaviour of  $J_{\mathrm{max}}(t)$ when $t\to +\infty $: We get
$$\sqrt{1+(rt)^2}=rt+{\cal O}\left( \frac{1}{rt}
\right),\
\alpha _+=r-\frac{1}{t}+{\cal O}\left(\frac{1}{trt} \right),
$$
\begin{equation}\label{++.8}
J_{\mathrm{max}}(t)=\left(1+{\cal O}\left(\frac{1}{rt} \right)
\right)\frac{2}{e}e^{rt}r,\ rt\to +\infty .
\end{equation}
Here we recall that the new $t=\widetilde{t}$ is equal to $t-a-b$ for
the original $t$. Returning to the original $I(\phi ,\psi
)=I_{a,b,t}(\phi ,\psi )$ (cf.\ (\ref{++.3})) with ${{\phi }_\vert}_{[0,a]}$, ${{\psi
  }_\vert}_{[0,b]}$ prescribed, we get with the choice
\begin{equation}\label{++.9}
\begin{cases}
\phi (s)-\phi (a)=\alpha _+(s-a),\ s\ge a,\\
\psi (s)-\psi (b)=\alpha _+(s-b),\ s\ge b,
\end{cases}
\end{equation}
\begin{equation}\label{++.10}
\alpha
_+=\alpha _+(\widetilde{t}),\ \ \widetilde{t}=t-a-b, \end{equation}
that 
\begin{equation}\label{++.11}
  \begin{split}
    &I(\phi ,\psi )=J_{\mathrm{max}}(t-a-b)e^{\phi (a)+\psi (b)}\\
    & =\left(1+{\cal O}\left(\frac{1}{t-a-b} \right)
\right) \frac{2}{e} e^{\phi (a)+\psi (b)}
\, r \, e^{r(t-a-b)},\ \ t-a-b\to +\infty ,
\end{split}
\end{equation}
when $r=r(\omega )>0$ is fixed.

\par Summing up the discussion so far, we get from (\ref{++.1}),
(\ref{++.11}):
\begin{prop}\label{++2}
Let $a,b>0$ and let $\Phi \in C([0,a]; \mathbb R)$, $\Psi \in C([0,b])$
be increasing, piecewise $C^1$ with $\Phi (0)=\Psi (0)=0$,
$$
\begin{cases}
r(\omega )^2\Phi ^2-\Phi '^2\le 0,\hbox{ on }[0,a],\\
r(\omega )^2\Psi ^2-\Psi '^2\le 0,\hbox{ on }[0,b].
\end{cases}
$$
Write $\Phi = e^\phi $, $\Psi =e^\psi $, with $\phi $, $\psi $
real. Then for $t>a+b$, with $r=r(\omega )$,
\begin{equation}\label{++.12}
e^{-\omega t}\| S(t)\|\le \frac{\| (r^2\Phi ^2-\Phi
  '^2)^{\frac{1}{2}}_-m\|_{e^{\omega \cdot }L^2([0,a])}\| (r^2\Psi ^2-\Psi
  '^2) ^{\frac{1}{2}}_-m\|_{e^{\omega \cdot }L^2([0,b])}}{J_{\mathrm{max}}(t-a-b)\Phi
(a)\Psi (b)},
\end{equation}
where $J_{\mathrm{max}}(\widetilde{t})$ is given in (\ref{++.7.5}) and
has the asymptotics (\ref{++.8}). In particular for large values of 
$t-a-b$, 
\begin{equation}\label{++.13}
\begin{split}
e^{-\omega t}\| S(t)\|\le & \left(1+{\cal
  O}\left(\frac{1}{r(t-a-b)} \right) \right) \frac{e}{2r}e^{-r(t-a-b)}\times \\
& \frac{\| (r^2\Phi ^2-\Phi
  '^2) ^{\frac{1}{2}}_-m\|_{e^{\omega \cdot }L^2([0,a])}\| (r^2\Psi ^2-\Psi
  '^2) ^{\frac{1}{2}}_-m\|_{e^{\omega \cdot }L^2([0,b])}}{\Phi
(a)\Psi (b)}.
 \end{split}
\end{equation}
\end{prop}
Here we meet the same quantities as in the previous
  section. Hence we obtain (cf Theorem \ref{prop4.9}), if $a^{*} (m)$
  is bounded,  $r=1$, $\omega=0$ and  $a, b \leq a^{*}  (m)$, $t> a+b$, 
 \begin{equation}
 || e^t S(t)|| \leq   \left(1+{\cal
  O}\left(\frac{1}{(t-a-b)} \right) \right) \frac{e}{2} m(a) m(b) e^{a+b}  \psi_0(a)^\frac 12  \psi_0 (b)^\frac 12  \,.
\end{equation}
As $t\rightarrow +\infty$, we have lost a factor $(e/2)$ in comparison with the statement of Theorem \ref{prop4.9}. However it is not excluded that for  some $t$ the estimate obtained by this approach is better.

\paragraph{Non optimality. Possible improvements?}
We have solved the optimization problem for $I(\phi ,\psi )$ in
(\ref{++.4}) for $(\phi ,\psi )$ varying in a restricted class. The
purpose of this remark is to show that the solution $(\phi ,\psi )$ in
(\ref{++.5}) with $\alpha =\alpha _+$ is not a critical point for
$I(\phi ,\psi )$ when $(\phi ,\psi )$ varies more freely and hence we
can perturbe our special solution slightly (leaving the restriced
class) and find an even larger value of $I(\phi ,\psi )$.

\par Write $f=\iota \psi $ for simplicity. We then want to find $\phi
,f\in C^2([0,1])$ with
\begin{equation}\label{++1.1}
\phi (0)=f(t)=0,
\end{equation}
$\phi$ increasing, $f$ decreasing (i.e.\ $\phi '\ge 0$, $f'\le 0$)
with
\begin{equation}\label{++1.2}
r^2-\phi '^2>0,\ \ r^2-f'^2>0,
\end{equation}
such that $I(\phi ,\iota f)$ is as large as possible and in particular
such that $(\phi ,f)$ is a critical point for $I$. We make a
variational calculation considering infinitessimal variations
$(\phi +\delta \phi ,f+\delta f)$ with
$\delta \phi (0)=\delta f(t)=0$. Then
$$
\delta I(\phi ,\iota f)=\mathrm{I}+\mathrm{II}+\mathrm{III},
$$
where with
$
K(\phi ,f)(s):=e^{\phi +f}(r^2-\phi '^2)^{\frac{1}{2}}(r^2-f '^2)^{\frac{1}{2}}:
$
$$
\mathrm{I}=\int_0^t K(\phi ,f)(s)(\delta \phi (s)+\delta f(s))ds,
$$
$$
\mathrm{II}=\int_0^t K(\phi ,f)(s) \frac{\delta \left( (r^2-\phi
    '^2)^{1/2} \right)}
{(r^2-\phi '^2)^{1/2}} ds,
$$
$$
\mathrm{III}=\int_0^t K(\phi ,f)(s) \frac{\delta \left( (r^2-f
    '^2)^{1/2} \right)}
{(r^2-f'^2)^{1/2}} ds.
$$
Here,
$$
\delta \left( (r^2-\phi '^2)^{1/2} \right)=-(r^2-\phi '^2)^{-1/2}\phi
'\delta \phi '
$$
and similarly for $f$, so
$$
\mathrm{II}=-\int_0^t K(\phi ,f)(s)(r^2-\phi '^2)^{-1}\phi '\delta
\phi ' ds,
$$
$$
\mathrm{III}=-\int_0^t K(\phi ,f)(s)(r^2-f '^2)^{-1}f '\delta
f ' ds.
$$
Here we integrate by parts, using that $\delta \phi (0)=\delta
f(t)=0$:
$$
\mathrm{II}=\int_0^t \left( \partial _s\circ K(\phi ,f)(r^2-\phi
  '^2)^{-1}\circ \partial _s\phi  \right)\delta \phi ds
-K(\phi ,f)(r^2-\phi '^2)\phi '\delta \phi (t),
$$
$$
\mathrm{III}=\int_0^t \left( \partial _s\circ K(\phi ,f)(r^2-f
  '^2)^{-1}\circ \partial _sf  \right)\delta f ds
+K(\phi ,f)(r^2-f '^2)f '\delta f (0).
$$
This gives,
\[
  \begin{split}
\delta (\phi ,\iota f)=&\int_0^t\left( K(\phi ,f)+\partial _s\circ K(\phi ,f)(r^2-\phi
  '^2)^{-1}\circ \partial _s\phi  \right) \delta \phi ds\\
&-K(\phi ,f)(r^2-\phi '^2)\phi '\delta \phi (t)+\\
&\int_0^t\left( K(\phi ,f)+\partial _s\circ K(\phi ,f)(r^2-f
  '^2)^{-1}\circ \partial _sf  \right) \delta f ds\\
&+K(\phi ,f)(r^2-f '^2)f'\delta f (0).
\\
  \end{split}
\]
The Assumption (\ref{++1.2}) implies that
$K(\phi ,f)(r^2-\phi '^2)>0$, $K(\phi ,f)(r^2-f '^2)>0$ and we see
that $(\phi ,f)$ is a critical point precisely when
\begin{equation}\label{++1.3}
\phi '(t)=0,\ \ f'(0)=0,
\end{equation}
(in addition to (\ref{++1.1})) and
\begin{equation}\label{++1.4}
  \begin{cases}
K(\phi ,f)+\partial _s\circ K(\phi ,f)(r^2-\phi
  '^2)^{-1}\circ \partial _s\phi=0\\
K(\phi ,f)+\partial _s\circ K(\phi ,f)(r^2-f
  '^2)^{-1}\circ \partial _sf=0
\end{cases} \hbox{ on }[0,t].
\end{equation}
We conclude that $(\phi ,\psi )$ in (\ref{++.5}) with $\alpha =\alpha
_+$, is not a critical point for $I(\cdot ,\cdot \cdot )$, since it
does not satisfy (\ref{++1.2}). Hence by modifying $\phi $ slightly
near $s=t$, we can increase $I(\phi ,\psi )$ further.\\

{\bf Acknowledgements.}\\
The IMB receives support from the EIPHI Graduate School (contract ANR-17-EURE-0002).\\


\begin{thebibliography}{99}






\bibitem{BuZw} N.~Burq, M.~Zworski. \newblock
 Geometric control in the presence of a black box. \newblock
  J. Amer. Math. Soc.  17(2) (2004), 443-471. 
  
  \bibitem{CST} R. Chill, D. Seifert, and Y. Tomilov.
  \newblock Semi-uniform stability of operator semi-groups and energy decay of damped waves.
  \newblock Philosophical Transactions A. The Royal Society Publishing. July 2020.
  




\bibitem{Da07} E.B.~Davies.  \newblock {\it Linear operators and their spectra,}
   \newblock Cambridge Studies in Advanced Mathematics, 106. Cambridge University
  Press, Cambridge, 2007.



\bibitem{EnNa07} K.J.~Engel, R.~Nagel.  \newblock {\it One-parameter semigroups
  for linear evolution equations,}  \newblock Graduate
  Texts in Mathematics, 194. Springer-Verlag, New York, 2000.

\bibitem{EnNa2} K.J.~Engel, R.~Nagel.
 \newblock {\it A short course on operator semi-groups},  \newblock  Unitext, Springer-Verlag (2005).


 \bibitem{GGN} I. Gallagher, T. Gallay and F. Nier.
\newblock  Spectral asymptotics for large skew-symmetric perturbations of
the harmonic oscillator,
 \newblock   Int. Math. Res. Not. IMRN 2009, no. 12, 2147--2199. 

 
  
  \bibitem{He1} B. Helffer.
  \newblock Spectral Theory and its Applications.
\newblock Cambridge University Press (2013). 
 


 
 
 \bibitem{HelSj} B.~Helffer and J.~Sj\"ostrand.
 \newblock From resolvent bounds to semigroup bounds.
 \newblock ArXiv:1001.4171v1, math. FA (2010). 







\bibitem{Hi} M. Hitrik.
 \newblock Eigenfunctions and expansions for damped wave equations.
 \newblock Meth. Appl. Anal. 10 (4) (2003), 1-22.

\bibitem{Paz} A. Pazy.
 \newblock {\it Semigroups of linear operators and applications to partial
  differential operators.}
 \newblock Appl. Math. Sci. Vol. 44, Springer (1983).


\bibitem{Sch09} E.~Schenk,  {\it Syst\`emes quantiques ouverts et
   m\'ethodes semi-classiques,} th\`ese novembre 2009.\\
    http://www.lpthe.jussieu.fr/~schenck/thesis.pdf

 \bibitem {Sj} J. Sj\"ostrand.  \newblock Resolvent estimates  for non-self-adjoint operators via
semi-groups.  \newblock Around the research of Vladimir Maz'ya. III, 359--384,
\newblock  Int. Math. Ser. (N. Y.), 13, Springer, New York, 2010. 


\bibitem{Sj2} J. Sj\"ostrand.  \newblock  Spectral properties for non
    self-adjoint differential operators.   \newblock Proceedings of the Colloque
  sur les \'equations aux d\'eriv\'ees partielles, \'Evian, June 2009,
  

  \bibitem{Sj3} J. Sj\"ostrand.
  \newblock {\it  Non self-adjoint differential operators, spectral asymptotics and random perturbations.}
\newblock  Pseudo-differential Operators and Applications. Birkh\"auser (2018).
  



\bibitem {TrEm} L.N.~Trefethen, M.~Embree.  \newblock {\it Spectra and
  pseudospectra. The behavior of nonnormal matrices and operators.}
  \newblock  Princeton University Press, Princeton, NJ, 2005.


  
  \bibitem{W} Dongyi Wei.
  \newblock Diffusion and mixing in fluid flow via the resolvent estimate.
  \newblock  Science China Mathematics, volume 64,  507--518 (2021).
  



\end{thebibliography}
\end{document}